\documentclass[a4paper,12pt]{article}

\makeatletter
\@addtoreset{footnote}{page}
\makeatother

\usepackage{style-enumitem}

\usepackage{amsthm}
\usepackage{amsmath,amssymb,latexsym,amsfonts,mathrsfs}

\renewcommand{\Im}{\mathop{\rm Im}}

\renewcommand{\tilde}{\widetilde}
\renewcommand{\bar}{\overline}
\newcommand{\un}[1]{\underline{#1}}

\newcommand{\absol}[1]{\left| #1 \right|} 
\newcommand{\norm}[1]{\left\| #1 \right\|} 
\newcommand{\rbra}[1]{\!\left( #1 \right)} 
\newcommand{\cbra}[1]{\!\left\{ #1 \right\}} 
\newcommand{\sbra}[1]{\!\left[ #1 \right]} 

\newcommand{\bC}{\ensuremath{\mathbb{C}}}
\newcommand{\bD}{\ensuremath{\mathbb{D}}}
\newcommand{\bE}{\ensuremath{\mathbb{E}}}

\newcommand{\bN}{\ensuremath{\mathbb{N}}}

\newcommand{\bP}{\ensuremath{\mathbb{P}}}

\newcommand{\bR}{\ensuremath{\mathbb{R}}}

\newcommand{\cA}{\ensuremath{\mathcal{A}}}
\newcommand{\cB}{\ensuremath{\mathcal{B}}}

\newcommand{\cD}{\ensuremath{\mathcal{D}}}
\newcommand{\cE}{\ensuremath{\mathcal{E}}}
\newcommand{\cF}{\ensuremath{\mathcal{F}}}

\newcommand{\cJ}{\ensuremath{\mathcal{J}}}

\newcommand{\cL}{\ensuremath{\mathcal{L}}}

\newcommand{\cT}{\ensuremath{\mathcal{T}}}

\newcommand{\bfF}{\ensuremath{{\mathbf{F}}}}

\newcommand{\bfH}{\ensuremath{{\mathbf{H}}}}
\newcommand{\bfI}{\ensuremath{{\mathbf{I}}}}

\newcommand{\bfL}{\ensuremath{{\mathbf{L}}}}

\newcommand{\bfP}{\ensuremath{{\mathbf{P}}}}
\newcommand{\bfQ}{\ensuremath{{\mathbf{Q}}}}
\newcommand{\bfR}{\ensuremath{{\mathbf{R}}}}

\newcommand{\bfT}{\ensuremath{{\mathbf{T}}}}

\newcommand{\bfV}{\ensuremath{{\mathbf{V}}}}
\newcommand{\bfW}{\ensuremath{{\mathbf{W}}}}

\newcommand{\bfb}{\ensuremath{{\mathbf{b}}}}

\newcommand{\bff}{\ensuremath{{\mathbf{f}}}}
\newcommand{\bfg}{\ensuremath{{\mathbf{g}}}}

\newcommand{\bfn}{\ensuremath{{\mathbf{n}}}}

\newcommand{\bft}{\ensuremath{{\mathbf{t}}}}

\newcommand{\ttd}{\ensuremath{{\mathtt{d}}}}

\newcommand{\ttn}{\ensuremath{{\mathtt{n}}}}

\newcommand{\ttt}{\ensuremath{{\mathtt{t}}}}

\newcommand{\ttE}{\ensuremath{{\mathtt{E}}}}

\newcommand{\diff}{\ensuremath{{\mathrm{d}}}}

\theoremstyle{plain}
\newtheorem{Thm}{Theorem}[section]

\newtheorem{Lem}[Thm]{Lemma}

\theoremstyle{definition}
\newtheorem{Ass}[Thm]{Assumption}

\newtheorem{Rem}[Thm]{Remark}

\numberwithin{equation}{section}

\makeatletter
\renewcommand\section{\@startsection {section}{1}{\z@}%
                                   {-3.5ex \@plus -1ex \@minus -.2ex}%
                                   {2.3ex \@plus.2ex}%
                                   {\normalfont\large\bf}}
\makeatother

\makeatletter
\renewcommand\subsection{\@startsection {subsection}{1}{\z@}%
                                   {-3.5ex \@plus -1ex \@minus -.2ex}%
                                   {2.3ex \@plus.2ex}%
                                   {\normalfont\normalsize\bf}}
\makeatother

\usepackage{color}
\allowdisplaybreaks[3]
\usepackage{mathtools}
\mathtoolsset{showonlyrefs=true}

\begin{document}

\begin{center}
{\Large \bf 
Scale operators for Cram\'er--Lundberg type Markov additive processes 
}
\end{center}
\begin{center}
Kei Noba
\end{center}

\begin{abstract}
In this paper, we consider Markov additive processes (MAPs) whose additive components are $\bR$-valued and have no positive jumps.
Specifically, we assume that the modulators are Feller processes with only finitely many jumps on each bounded time interval and that the additive components also have only finitely many jumps on each bounded time interval.
For such MAPs, we characterize the associated scale operators and use them to solve the two-sided exit problem, which concerns the first entrance time into one half-line on the event that it is reached before the other, and to characterize the potential measures of the processes killed when their additive components exit an interval.
Our proofs rely on properties of analogues of local times, properties of the exit systems obtained by pairing them with suitable kernels, and results on the extension of $C_0$-semigroups to $C_0$-groups.
\end{abstract}

\section{Introduction}
Characterizing the distributions of hitting times of sets and potential measures is an important problem in the study of $\bR$-valued L\'evy processes.
Here we focus on first passage times into half-lines, particularly the two-sided exit problem, which concerns the Laplace transform of the first entrance time into one half-line restricted to the event that it is reached before the other.
We also consider the potential measures of processes killed upon exiting an interval.
A fundamental tool is the Wiener--Hopf factorization (see, for example, \cite[Section 6]{Kyp2014}), which characterizes the distributions of the supremum and infimum of a L\'evy process up to an independent exponential time.
This factorization provides a starting point for the analysis of two-sided exit problems and the characterization of entrance distributions and potential measures for meromorphic L\'evy processes in \cite{KuzKypPar2012}.
It also underlies the characterization in \cite{KypWat2014} of the $0$-potential measure of a process killed upon exiting an interval.
In this paper, we are particularly interested in spectrally negative L\'evy processes, that is, L\'evy processes with no positive jumps whose paths are not monotone.
Each such process has associated functions called scale functions, in terms of which two-sided exit identities and the potential measures of the process killed upon exiting an interval can be expressed (see \cite[Section 8]{Kyp2014} or \cite{KuzKypRiv2012} for background on scale functions).
The theory of scale functions can be developed using Wiener--Hopf factors, as in \cite[Theorem 8.1]{Kyp2014}, or through potential-theoretic arguments, as in \cite{Pis2005}.
This theory has been extended to various $\bR$-valued Markov processes with no positive jumps.
In particular, \cite{Nob2020a} defines scale functions for such processes using excursion measures and establishes their fundamental properties through excursion theory and properties of local times.
As discussed in \cite[Section 10]{Kyp2014} and \cite[Section 1.2]{KuzKypRiv2012}, scale functions have important applications in stochastic control, risk theory, queueing theory, and other areas.
\par
In this paper, we study Markov additive processes (MAPs), which generalize L\'evy processes.
The general theory of MAPs was developed in \cite{Cin1972a, Cin1972b}.
A MAP is a Markov process consisting of an additive component $\{X_t:t\geq 0\}$ and a modulator $\{Y_t:t\geq 0\}$.
The modulator is itself a Markov process.
Here we assume that both the MAP and its modulator have the Feller property.
Informally, the additive component behaves like a L\'evy process whose L\'evy--Khintchine characteristics vary with the state of the modulator.
We assume that the additive component is $\bR$-valued, has no positive jumps, and has paths that are not monotone; we refer to such MAPs as spectrally negative MAPs.
Our aim is to solve the two-sided exit problem for the additive component and characterize the potential measures of the MAP killed when its additive component exits an interval.
When the modulator is a Markov chain, the theory of scale matrices was developed in \cite{KypPal2008} and \cite{IvaPal2012} to address these problems.
In particular, \cite{IvaPal2012} defines scale matrices in terms of local times and develops their theory by exploiting this representation together with matrix-theoretic properties.
\par
Our aim is to develop a theory of scale functions for spectrally negative MAPs whose modulators are Feller processes that may evolve continuously.
In this setting, the corresponding objects are operators, which we call scale operators.
We restrict our attention to the case in which the modulator has only finitely many jumps on each bounded time interval, while the additive component also has only finitely many jumps on each bounded time interval and has paths of bounded variation.
Since this class generalizes Cram\'er--Lundberg processes, we use the term ``Cram\'er--Lundberg type MAPs'' in the title.
As in the existing theory, we characterize scale operators through their Laplace transforms and express both two-sided exit identities and the potential measures of the MAP killed when its additive component exits an interval in terms of scale operators.
The key ingredients of our proofs are two complementary descriptions of scale operators, one in terms of local times and the other in terms of exit systems, which generalize excursion measures, together with the extendibility of the relevant $C_0$-semigroups to $C_0$-groups.
In particular, the restrictions on jump activity described above allow us to apply perturbation theorems for generators to preserve this extendibility.
The use of these techniques is a major difference from previous approaches.
This group extension provides the invertibility needed in our operator setting, a property readily available in earlier work on spectrally negative L\'evy processes and spectrally negative MAPs with Markov chain modulators.
\par
The remainder of this paper is organized as follows.
Section 2 introduces the MAPs considered in this paper, reviews relevant known results, and presents the constructions of the local times and exit systems used in our analysis.
Section 3 presents the main results, and Section 4 provides their proofs.

\section{Preliminaries}
\subsection{Cram\'er--Lundberg type Markov additive processes}
Let $E$ be a locally compact, separable metric space. 
Let $(X, Y)=(\Omega, \cF, \cF_t, (X_t, Y_t),\theta_t, \bP_{(x,y)})$ be an $\bR\times E$-valued Feller process with no killing (for the definition of Feller processes, see, e.g., \cite[Chapter 17]{Kal2021}). 
Here, $X_t$ takes values in $\bR$ and $Y_t$ takes values in $E$. 
We write $\cA$ for the generator of $(X, Y)$ on $C_0(\bR\times E)$. 
We assume that 
$(X, Y)$ satisfies, for non-negative measurable function $f$ and $(x,y)\in\bR\times  E $, 
\begin{align}
\bE_{(x,y)}\sbra{f(X_t, Y_t)}=\bE_{(0,y)}\sbra{f(x+X_t, Y_t)}. \label{additive_property}
\end{align}
Then, the process $(X, Y)$ is called a MAP. 
Note that $ Y=(\Omega, \cF, \cF_t, Y_t,\theta_t, \bP^Y_{y})$ where $ \bP^Y_{y}:= \bP_{(x,y)}$ is also a Feller process. 
\par
For a locally compact, separable metric space $S$, we write $C_0(S)$ and $C_b(S)$ for the sets of continuous functions vanishing at infinity and bounded continuous functions, respectively. 
For a generator $\cL$, we write $D(\cL)$ for the domain of $\cD$. 
\par
We impose the following assumptions on $Y$. 
\begin{Ass} \label{Ass211}
There exist a $E$-valued Feller process $ Y^\prime=(\Omega^\prime, \cF^\prime, \cF^\prime_t, Y^\prime_t,\theta^\prime_t, \bP^{Y^\prime}_{y})$ and a kernel $\nu$ from $E$ to $E$ satifying the following conditions. 
\begin{enumerate} 
\item It holds $\sup_{y\in E}\nu(y, E)<M_\nu$ for some $M_\nu>0$. 
In addition, for $f\in C_b(E^2)$ such that $\lim_{y^\prime \to\infty} \sup_{y\in E}\absol{f(y, y^\prime)}=0$, 
\begin{align}
y\mapsto \int_E f(y,  y^\prime)\nu(y, \diff y^\prime)\in C_0(E). 
\end{align}
\item 
There exists a map $(y, t)\mapsto \phi_t(y)$ from $E\times [0,\infty)$ to $E$ such that $t\mapsto \phi_t(y)$ is continuous for each $y\in E$, $y\mapsto \phi_t(y)$ is continuous for each $t\geq 0$, and, for every $y\in E$, $Y_t=\phi_t(y)$ for all $t\geq 0$, $\bP^{Y^\prime}_y$-almost surely.
In addition, the semigroup $Y^\prime$ acting on $C_0(E)$ can be extended to a $C_0$-group $\{\bfP^{Y^\prime}_t:t\geq 0\}$. 
Consequently, $C_0$-group $\{\bfP^{Y^\prime}_t:t\geq 0\}$ consists of surjective isometries.
\item 
We define an operator $\cA_\nu$ as 
\begin{align}
\cA_\nu f (y)=\int_E (f(y_1)-f(y))\nu(y, \diff y_1),\qquad y\in E, 
\end{align}
for bounded measurable function $f$. 
Then, $\cA_\nu C_0(E) \subset C_0(E)$. 
We have
\begin{align}
\cA_Yf (y)=\cA_{Y^\prime}f (y)+\cA_\nu f(y),\qquad y\in E, 
\end{align}
for $f\in D(\cA_{Y})(=D(\cA_{Y^\prime}))$, where $\cA_{Y}$ and $\cA_{Y^\prime}$ are generators of $Y$ and $Y^\prime$, respectively. 
\end{enumerate}
\end{Ass}
In other words, $Y$ can be regarded as a process obtained from $Y^\prime$ by adding jumps governed by $\nu$ (see also Lemma \ref{LemB01}). 
Since the corresponding jump operator is a bounded perturbation of the generator of $Y^\prime$, \cite[p.79 and Theorem III.1.3]{EngNag2000} implies that the semigroup of $Y$ also extends to a $C_0$-group $\{\bfP^Y_t:t\geq 0\}$ on $C_0(E)$. 
\begin{Rem}\label{Rem202}
Let us explain why $Y^\prime$ is required to be the deterministic process specified in Assumption \ref{Ass211}. 
We need to use the main theorem of \cite{Dor1966} in the proof of Lemma \ref{Lem502}. 
For this purpose, the $C_0$-group $\{\bfP^{Y^\prime}_t:t\geq 0\}$ must be contractive and hence consists of surjective isometries.
Then, we can apply \cite[Theorem 7.1]{Beh1979} to $\bfP^{Y^\prime}_t$, and then
there exists a continuous map $\phi_t$ from $E$ to $E$ such that $\bfP^{Y^\prime}_tf(y)=f(\phi_t(y))$.
Since $Y^\prime$ is a Feller process, it is strongly continuous for non-negative times. Strong continuity for negative times follows from
\begin{align}
\lim_{t\downarrow0}\norm{\bfP^{Y^\prime}_{-t}f-f}_\infty
=\lim_{t\downarrow0}\norm{\bfP^{Y^\prime}_{t}\rbra{\bfP^{Y^\prime}_{-t}f-f}}_\infty
=\lim_{t\downarrow0}\norm{f-\bfP^{Y^\prime}_tf}_\infty=0.
\end{align}
Consequently, 
the map $t\mapsto\phi_t(y)$ is continuous. Thus, $Y'$ is a deterministic process with continuous paths.
\end{Rem}
By Remark \ref{Rem202}, \cite[Theorem 5.8]{Bas1979} and \cite[p.346]{Sha1988}, we have the following lemma.  
\begin{Lem}\label{LemB01}
For $q>0$, $y\in E$, $T\geq 0$, non-negative predictable process $\{Z_t:t\geq 0\}$ and non-negative measurable function $f$ on $\bR^2$, we have 
\begin{align}
\bE^Y_y \sbra{\sum_{t\in \bfT_\nu \cap[0, T]}Z_tf(Y_{t-}, Y_t)}=\bE^Y_0\sbra{\int_0^T Z_t
\rbra{\int_E f(Y_t , y) \nu(Y_t, \diff y)}\diff t},
\end{align}
where $\bfT_\nu= \{t\geq 0: Y_{t-}\neq Y_t \}$. 
\end{Lem}
\par 
By \cite[Propositiom 2.20, Theorem 2.22 and Corollary 2.25]{Cin1972b}, it is known that for $\bP_{(x, y)}$, there exists the set of regular probability $\bP^\omega_x$, which represent the law of $X=\{X_t:t\geq 0\}$,with $\omega \in \Omega$ such that 
$X$ behaves as an additive process with some jumps such that its characteristic exponent is written using the path of $\{Y_t(\omega):t\geq 0\}$ under $\bP_{(x, y)}$.  
In particular, in this paper, we consider the case under the following assumption.
\begin{Ass}\label{Ass203}
The characteristic exponent of $X$ is given in the following form under $\bP_{(0, y)}$: for $t\geq 0$ and $\lambda \in\bR$, 
\begin{align}
\bP^\omega_{ 0}\sbra{e^{i\lambda X_t}}
=
&\rbra{\prod_{s\in[0,t]}F_{Y_{s-}(\omega),Y_s(\omega)}(i\lambda)}
\exp\cbra{ i\lambda D_t(\omega)+\int_{(-\infty , 0)} \rbra{e^{i\lambda x}-1 } N_t (\omega, \diff x)},
\label{L--K}
\end{align} 
where $F_{y_1, y_2}$, $D_t$ and $N_t$ satisfy the following conditions:
\begin{enumerate}
\item 
For $y_1, y_2\in E$, $\rho_{y_1, y_2}$ is the probability measure on $(-\infty, 0]$. 
For $y_1, y_2\in E$ with $y_1\equiv y_2$, $\rho_{y_1, y_2}\equiv\delta_0$. 
In addition, 
\begin{align}
(y_1, y_2)\mapsto
\int_{(-\infty, 0)} f(y_2, x)
\rho_{y_1, y_2}(\diff x)\in C_b(E^2), \qquad f\in C_b(E\times (-\infty, 0]). \label{finiteness1}
\end{align}
For $y_1, y_2\in E$ with $y_1\neq y_2$, 
\begin{align}
F_{y_1, y_2} (\gamma)=\int_{(-\infty, 0]} e^{\gamma x} \rho_{y_1, y_2}(\diff x),\qquad \gamma
\in\{z\in\bC:\Im z\geq 0\}. 
\end{align}
\item There exists a continuous function $\ttd$ on $E$ such that 
\begin{align}
\ttd (y) \in (\un{d}, \bar{d}),\qquad y\in E ,
\end{align}
for some $\un{d}, \bar{d}>0$, and
\begin{align}
D_t(\omega)=\int_0^t \ttd(Y_s(\omega)) \diff s,\qquad t\geq0. 
\end{align} 
\item There exists a kernel $\ttn$ from $ E $ to $(-\infty, 0)$ such that
$\sup_{y\in E}\ttn (y,\bR)<M_\ttn$ for some $M_\ttn>0$, 
\begin{align}
y\mapsto \bfn (f) (y):=\int_\bR f(y, x) \ttn(y, \diff x) \in C_b (E), \qquad f\in C_b(E\times (-\infty ,0] ), \label{finiteness2}
\end{align}
and
\begin{align}
N_t (\omega, \diff x) = \int_0^t \ttn(Y_s(\omega), \diff x)\diff s,\qquad t\geq 0. 
\end{align} 
\end{enumerate}
\end{Ass}
Then, the map $t\mapsto X_t$ has no positive jumps, $\bP_{(x, y)}$-a.s. for all $(x, y)\in\bR\times  E $ and then 
$\{x\}\times  E $ is a thin set but not a polar set for all $x\in\bR$. 
\par
As stated in the Introduction, one of the aims of this paper is to characterize the distributions of the following two hitting times:
\begin{align}
\tau^+_x:=\inf\{t>0:X_t>x\},\qquad  \tau^-_x:=\inf\{t>0:X_t<x\},\qquad x\in\bR. 
\end{align}
\begin{Rem}\label{Rem205}
By \cite[Theorem 3.8]{Kyp2014}, 
the expression obtained by substituting $\lambda=-i\beta$ with $\beta\geq 0$ into \eqref{L--K} is also well defined. 
In the case of a spectrally negative L\'evy process, this corresponds to what is called the Laplace exponent.
\end{Rem}
\begin{Rem}\label{regular}
Since $X$ has a positive drift and only finitely many jumps on every finite time interval, and does not jump at time $0$ by right-continuity, $(x,y)\in\bR\times E$ is regular for $(x,\infty)\times E$. Let
$T_a^+:=\inf\{t\geq 0:X_t\geq a\}$.
Then, by the strong Markov property at $T_a^+$, we have
$T_a^+=\tau_a^+$, $\bP_{(x,y)}$-a.s.
By quasi-left continuity of $(X, Y)$, $\lim_{\varepsilon\downarrow}X_{\tau^+_{a-\varepsilon}}=a$, so combining this with the preceding arguments 
$\lim_{\varepsilon\downarrow}\tau^+_{a-\varepsilon}=\tau^+_a$. 
\end{Rem}

\subsection{Occupation density formula and exit system}
In this subsection, we define a family of additive functionals that serve as local times and satisfy an occupation density formula, together with the associated exit systems.
\par
For $x\in\bR$, we define 
\begin{align}
\ell^x_t = \text{Card}\{s\in[0,t]: X_s=x\},\qquad t\geq 0.
\end{align} 
By \cite[p.28]{GemHor1980}, the map $(s, x, \omega) \mapsto \ell^x_t(\omega)$ on $[0, t]\times \bR\times \Omega$ is universally measurable. 
For $x\in\bR$, we also define a process $L^x=\{L^x_t:t\geq 0\}$ by
\begin{align}
L^x_t = \int_{[0,t]}\frac{1}{\ttd(Y_s)}\diff \ell^x_s,\qquad t\geq 0.
\end{align} 
\begin{Lem}\label{Lem301}
For non-negative measurable function $f$, we have 
\begin{align}
\int_0^t f(X_t)  \diff t =\int_\bR f(x) L^x_t\diff x,
\end{align}
$\bP^\omega_y$-a.s. 
\end{Lem}
\begin{proof}
We fix $\omega\in\Omega$ and $y\in  E $ and consider the occupation density formula of $X$ under $\bP^\omega_y$. 
For simplicity, we shall write $y_t$ for $Y_t(\omega)$ for $t\geq 0$. 
Then, by \cite[Theorem 1]{BerYor2014}, we have, for non-negative measurable function $f$ and $t\geq 0$,
\begin{align}
\int_0^t f(X_t) \ttd(y_t) \diff t =\int_\bR f(x) \ell^x_t\diff x. \label{aaa}
\end{align}
$\bP^\omega_y$-a.s. 
Since $t\mapsto\ttd(y_t)$ is c\`adl\`ag and has finite jumps on $[0, t]$ by Lemma \ref{LemB01} and Assumption \ref{Ass203} (ii)
, and by \eqref{aaa} and the monotone convergence theorem, we have 
\begin{align}
\int_0^t f(X_t)  \diff t
&=\lim_{n\to\infty} \sum_{k=1}^{2^n}\int_{\frac{k-1}{2^n}}^{\frac{k}{2^n}} f(X_t) \frac{\ttd(y_t)}{\sup_{s\in(\frac{k-1}{2^n}, \frac{k}{2^n}]}\ttd(y_s)} \diff t\\
&=\lim_{n\to\infty} \sum_{k=1}^{2^n}\int_\bR f(x)\frac{\ell^x_{\frac{k}{2^n}}-\ell^x_{\frac{k-1}{2^n}}}{\sup_{s\in(\frac{k-1}{2^n}, \frac{k}{2^n}]}\ttd(y_s)}\diff x\\
&=\int_\bR f(x)\lim_{n\to\infty} \sum_{k=1}^{2^n}\frac{\ell^x_{\frac{k}{2^n}}-\ell^x_{\frac{k-1}{2^n}}}{\sup_{s\in(\frac{k-1}{2^n}, \frac{k}{2^n}]}\ttd(y_s)}\diff x=\int_\bR f(x) L^x_t\diff x.
\end{align}
The proof is complete.
\end{proof}
For each $x\in\bR$, we need to introduce some notation. 
We define 
\begin{align}
T_x:=\inf\{t>0: X_t =x\}.
\end{align}
We also define $\bD_x$ to be the space of all c\`adl\`ag functions from $[0,\infty)$ to $\bR\times E$ that start from a point in $\{x\}\times E $
and, upon returning to  $\{x\}\times E $, remain at the return point thereafter. 
By a slight abuse of notation, we also use $(X_t, Y_t)$ to denote the canonical process on $\bD_x$, that is,
\begin{align}
(X_t(\omega), Y_t(\omega))=\omega(t), \qquad t\geq 0,\ \omega\in \bD_x.
\end{align}
We define $\cE^x$ by the kernel from $ E $ to $\bD_x$ such that for non-negative measurable function $F$ on $\bD_x$, 
\begin{align}
\cE^x_y\sbra{ F(\{(X_t, Y_t):t\geq 0\})}:=&\int_{\bD_x}F(\{(X_t(\omega), Y_t(\omega)):t\geq 0\})\cE^x(y, \diff \omega)\\
=&\ttd(y)\bE_{(x,y)}\sbra{ F(\{(X_{t\land T_x}, Y_{t\land T_x}):t\geq 0\})}.
\end{align} 
\begin{Lem}\label{Lem302}
For non-negative optional process $\{Z_t:t\geq 0\}$, non-negative measurable function $F$ and $(x, y)\in\bR\times E$, we have 
\begin{align}
&\bE_{(x,y)}\sbra{\sum_{t \in G^x}Z_t  \Big{(}F(\{(X_{t\land T_x}, Y_{t\land T_x}):t\geq 0\})\circ\theta_t\Big{)}}\\
&\qquad\qquad\qquad\qquad\qquad=\bE_{(x,y)}\sbra{\int_{[0, \infty)}Z_s \cE^x_{Y_s}\sbra{F(\{(X_t, Y_t):t\geq 0\})}\diff L^x_s}.
\end{align}
where $G^x=\{t\geq0: X_t=x\}$. 
\end{Lem}
It follows from \eqref{L--K} that $X$ can be decomposed into the sum of a drift component whose rate always lies between $\un{d}$ and $\bar{d}$, and a jump component having only finitely many jumps on every bounded time interval. Consequently, the points of $G^x$ in each bounded interval can be enumerated by finitely many stopping times. 
The lemma above then follows readily by applying the strong Markov property at these stopping times, and we therefore omit the proof. 
A pair $(\cE^x,L^x)$ satisfying the identity given in Lemma \ref{Lem302} is called an exit system away from $\{x\}\times E$.

\section{Main results
}
For $\beta\geq 0$, we define the operator $\bfF^{(\beta)}$ from $D(\cA_Y)$ to $C_0(E)$ by, for $f\in D(\cA_Y)$ and $y\in E$,  
\begin{align}
\bfF^{(\beta)}f(y)=
&\rbra{\beta \ttd(y)+\int_{(-\infty , 0)} \rbra{e^{\beta x}-1 } \bfn (y, \diff x)}f(y)\\
&\qquad\qquad+\int_{(-\infty, 0)}\rbra{F_{y,y^\prime}(\beta)f(y^\prime)-f(y)}\nu(y, \diff y^\prime)
+\cA_{Y^\prime} f(y), 
\end{align}
which is continuous by the assumptions. 
For $q, \beta\geq 0$, we define the operator $\bfF^{(\beta)}-q\bfI$ by 
\begin{align}
(\bfF^{(\beta)}-q\bfI)f(y)=\bfF^{(\beta)}f(y)-qf(y),\qquad f\in C_0(E), ~y\in E. 
\end{align}
Here, $\bfI$ denotes the identity operator. 
\begin{Thm}\label{Thm300}
For $q>0$ and $x\geq 0$, there exists an operator $\bfW^{(q)}_x$ from $C_0(E)$ to $C_0(E)$ satisfying the following conditions:
\begin{enumerate}
\item $\bfW^{(q)}_x$ is invertible. 
\item For $f\in C_0 (E)$, $(x, y)\mapsto \bfW^{(q)}_x f(y)$ is measurable. 
\item For all sufficiently large $\beta\geq 0$, the operator $\bfF^{(\beta)}-q\bfI$ is invertible and we have
\begin{align}
\int_0^\infty e^{-\beta x}\bfW^{(q)}_x f(y) \diff x=(\bfF^{(\beta)}-q\bfI)^{-1}f(y),\qquad f\in C_0(E), ~y\in E. 
\label{scale_Laplace_transform} 
\end{align}
\end{enumerate}
\end{Thm}

For $q>0$ and $a,b, x\in\bR$ with $a\leq x\leq b$, we define the operator $\bfQ^{(q),[a,b]}_x$, which represents the two-sided exit problem, by
\begin{align}
\bfQ^{(q),[a,b]}_x f(y)=&\bE_{(x, y)}\sbra{e^{-q\tau^+_b}f(Y_{\tau^+_b});\tau^+_b<\tau^-_a}, 
\qquad f\in C_0(E), ~ y\in E.
\end{align}
\begin{Thm}\label{Thm401}
For  $q>0$ and $a,b,x\in\bR$ with $a\leq x\leq b$, we have $\bfQ^{(q),[a,b]}_x= \bfW^{(q)}_{x-a}\bfW^{(q),-1}_{b-a}$. 
\end{Thm}
For $q>0$ and $a,b,x\in\bR$ with $a\leq x\leq b$, we define the operator $\bfL^{(q),[a,b]}_{\{x\}}$ by
and 
\begin{align}
\bfL^{(q),[a,b]}_{\{x\}}f(y)=\bE_{(0,y)}\sbra{\int_{(0, \tau^-_a\land \tau^+_b]}e^{-qs} f(Y_s)\diff L^x_s}, \qquad f\in C_0(E). 
\end{align}
By Lemma \ref{Lem301} and \eqref{additive_property}, we have, for a bounded measurable function $g$ on $R$ and $f\in C_0(E)$, 
\begin{align}
\bE_{(x,y)}\sbra{\int_{[0, \tau^-_a\land \tau^+_b]}e^{-qs} g(X_s)f(Y_s)\diff s}
=\int_\bR g(x^\prime) \bfL^{(q),[a-x,b-x]}_{\{x^\prime-x\}}f(y) \diff x^\prime. 
\end{align}
Thus, to characterize the resolvent measure, it suffices to characterize the operator $\bfL^{(q),[a,b]}_{\{x\}}$.
\begin{Thm}\label{Thm402}
For  $q>0$ and $a,b,x\in\bR$ with $a\leq x\leq b$, we have 
\begin{align}
\bfL^{(q),[a,b]}_{\{x\}}f(y)=&\bfW^{(q)}_{-a}\bfW^{(q),-1}_{b-a} \bfW^{(q)}_{b-x}f(y)
-
\bfW^{(q)}_{-x}f(y),\qquad f\in C_0(E),~y\in E, 
\end{align} 
where $\bfW^{(q)}_{x}f\equiv 0$ for $x<0$ and $f\in C_0(E)$. 
\end{Thm}

\section{Proofs}
\subsection{Extension of a Feller semigroup to a $C_0$-group}
For $q, a\geq 0$, we define the operator $\bfH^{(q)}_a$ as 
\begin{align}
\bfH^{(q)}_af(y) = \bE_{(0, y)}\sbra{e^{-q\tau^+_a}f (Y_{\tau^+_a}); \tau^+_a<\infty},\qquad y \in E,
\end{align}
for bounded measurable function $f$. 
\begin{Lem}\label{Lem401}
For $q>0$, the set $\bfH^{(q)}:=\{\bfH^{(q)}_a:a\geq 0\}$ is a Feller semigroup on $C_0(E)$. 
\end{Lem}
\begin{proof}
Before proceeding, we first verify a simple fact. 
Let ${\bf{e}}_q$ be an independent exponential random variable with rate $q>0$. 
We define $\{(X^{(q)}_t, Y^{(q)}_t) : t\geq 0\}$ by 
\begin{align}
(X^{(q)}_t, Y^{(q)}_t) = 
\begin{cases}
(X_t, Y_t),\qquad &t< {\bf{e}}_q, \\
\partial ,\qquad &t\geq {\bf{e}}_q,
\end{cases}
\end{align}
where $\partial$ is a cemetery point and $\tau^{(q)}_a=\inf\{t>0:X^{(q)}_t>a\}$ for $a\in\bR$. 
Then, $\{(X^{(q)}_t, Y^{(q)}_t):t\geq 0\}$ and $\{ Y^{(q)}_t:t\geq 0\}$ are also Feller processes with probability measura $\bP_{(x, y)}$ with $(x, y)\in\bR\times E$. 
We assume that $f(\partial)=0$ for $f\in C_0(E)$. 
Then, 
\begin{align}
\bfH^{(q)}_af(y) = \bE_{(0, y)}\sbra{f (Y^{(q)}_{\tau^{(q)}_a}); \tau^{(q)}_a<\infty}. 
\end{align}
From the fact above, it is easy to confirm that for $f\in C_0(E)$ and $\varepsilon>0$ there exists $\delta>0$ such that 
\begin{align}
\sup_{y\in E} \bE_{(0, y)}\sbra{\sup_{t\in[0,\delta]}|f(Y^{(q)}_t)-f(y)| }<\varepsilon. \label{supupper}
\end{align}
Note that \eqref{supupper} also holds for any $E$-valued Feller processes.  
We verify \eqref{supupper} here. 
Suppose that \eqref{supupper} does not hold for some $\varepsilon>0$. 
Then, there exists $n \in\bN$ such that $n\varepsilon>\norm{f}_\infty$. 
For $y\in E$, we define $\ttt^{\frac{\varepsilon}{2}}_y=\inf\{t>0: |f(Y_t)-f(y)|\geq \frac{\varepsilon}{2}\}$. 
Since $\{ Y^{(q)}_t:t\geq 0\}$ is a Feller process and \eqref{supupper} does not hold, we can take $\delta>0$ such that 
\begin{align}
\sup_{t\in[0,\delta]}\sup_{y\in E} \bE_{(0, y)}\sbra{|f(Y^{(q)}_t)-f(y)| }<\frac{\varepsilon}{8n},
\qquad \bP_{(0, y^\prime)}(\ttt^{\frac{\varepsilon}{2}}_{y^\prime}\leq \delta)\geq \frac{\varepsilon}{2\norm{f}_\infty},
 \label{supupperpre}
\end{align}
for some $y^\prime \in E$. 
By the strong Markov property at $\ttt^{\varepsilon}_y$ and \eqref{supupperpre}, we have, for $t\in[0,\delta]$, 
\begin{align}
\bE_{(0, y^\prime)}\sbra{|f(X_t)-f(y)|}&\geq \bE_{(0, y^\prime)}\sbra{\bE_{(Y^{(q)}_{\ttt^{\frac{\varepsilon}{2}}_y}, Y_{\ttt^{\frac{\varepsilon}{2}}_y})}
\sbra{|f(Y^{(q)}_s)-f(y)|}\big{|}_{s=t-\ttt^{\frac{\varepsilon}{2}}_y};\ttt^{\frac{\varepsilon}{2}}_y\leq t }\\
&\geq \bP_{(0, y^\prime)}\rbra{\ttt^{\frac{\varepsilon}{2}}_y\leq t }\rbra{\frac{\varepsilon}{2}-\frac{\varepsilon}{8n}}
\geq  \frac{\varepsilon}{2\norm{f}_\infty}\cdot\frac{3\varepsilon}{8}=\frac{3}{16n}\varepsilon, 
\end{align}
which contradicts the first inequality of \eqref{supupperpre}. 
\par
We verify that $\{\bfH^{(q)}_a:a\geq 0\}$ is a semigroup of positive contraction operators and satisfies conditions (F1) and (F2) in \cite[p.~369]{Kal2021}.
\par
For (F1) above, we show that functions in $C_0(E)$ are mapped into $C_0(E)$ by $\bfH^{(q)}_a$. For this, it suffices to
verify that $\bfH^{(q)}_a f\in C_0(E)$ for every non-negative function $f\in C_0(E)$ and $a \geq 0$ below. 
\par
(i) We verify that $\bfH^{(q)}_a f$ is a continuous function. 
We write $\bD
$ for the space of $(\bR \times E)\cup\{\partial\}$-valued c\`adl\`ag paths with $J_1$-topology. 
For $w=\{(w_1(t), w_2(t)):t\geq0\}\in \bD$, where $w_1(t)\in\bR$ and $w_2(t)\in E$, 
we denote $\tau^+_a(w)=\inf\{t>0:w_1(t)>a\}$. 
For $s, t\geq 0$ with $s\leq t$ and $b, c< 0$ with $b\leq c$, we define 
\begin{align}
\Xi_a (s, t; b, c)&=\{w\in \bD: \tau^+_a(w)\in [s, t), f(w_2(t-)), f(w_2(t))\in (b, c)\}, \\
\Xi_a (s, t; 0, 0)&=\{w\in \bD: \tau^+_a(w)\in [s, t), f(w_2(t-)), f(w_2(t))=0\}, \\
\Xi_a (\infty)&=\{w\in \bD: 
\sup_{t\in[0,\infty)}w_1(t)<a\},
\end{align}
Then, $\Xi_a (0, t; b, c)$ is a open set with respect to $J_1$-topology, so 
by \cite[Theorem 19.25]{Kal2021}, we have, for $y\in E$,  $s, t\geq 0$ with $s\leq t$ and $b, c\geq 0$ with $b\leq c$ such that $\bP_{(0, y)}(\tau^+_a=s)=\bP_{(0, y)}(\tau^+_a=t)=\bP_{(0, y)}(f(Y^{(q)}_{s-})=b)=\bP_{(0, y)}(f(Y^{(q)}_{s})=b)=\bP_{(0, y)}(f(Y^{(q)}_{t-})=c)=\bP_{(0, y)}(f(Y^{(q)}_{t})=c)=0$,
\begin{align}
\bP_{(0, y^\prime)}\rbra{\Xi_a (s, t; b, c)}
=&\bP_{(0, y^\prime)}\rbra{\Xi_a (0, t; b, c)}
-\bP_{(0, y^\prime)}\rbra{\Xi_a (0, s; b, c)}\\
\rightarrow&\bP_{(0, y)}\rbra{\Xi_a (0, t; b, c)}
-\bP_{(0, y)}\rbra{\Xi_a (0, s; b, c)}\\
=&\bP_{(0, y)}\rbra{\Xi_a (s, t; b, c)} \label{upestemate1}
\end{align}
\begin{align}
\bP_{(0, y^\prime)}\rbra{\Xi_a (s, t; 0, 0)}
=&\bP_{(0, y^\prime)}\rbra{\tau^{(q)}_a<t}-\bP_{(0, y^\prime)}\rbra{\Xi_a (0, t; 0, \infty)}\\
&-\bP_{(0, y^\prime)}\rbra{\tau^{(q)}_a<s}+\bP_{(0, y^\prime)}\rbra{\Xi_a (0, s; 0, \infty)}\\
\rightarrow&
\bP_{(0, y)}\rbra{\tau^{(q)}_a<t}-\bP_{(0, y)}\rbra{\Xi_a (0, t; 0, \infty)}\\
&-\bP_{(0, y)}\rbra{\tau^{(q)}_a<s}+\bP_{(0, y)}\rbra{\Xi_a (0, s; 0, \infty)}\\
=&\bP_{(0, y)}\rbra{\Xi_a (s, t; 0, 0)}, \label{upestemate4}
\end{align}
and 
\begin{align}
\bP_{(0, y^\prime)}\rbra{
\tau^{(q)}_a<\infty
}
=\bP_{(0, y^\prime)}\rbra{\Xi_a (\infty)}
\rightarrow 
\bP_{(0, y)}\rbra{\Xi_a (\infty)}
=
\bP_{(0, y)}\rbra{
\tau^{(q)}_a<\infty
},\label{asdf}
\end{align}
as $y^\prime\to y$, where $\bP_{(0, y)}\rbra{\Xi_a (~\cdot~)}:=\bP_{(0, y)}\rbra{\{(X^{(q)}_t, Y^{(q)}_t):t\geq 0\}\in\Xi_a (~\cdot~)}$ for $y\in E$. 
We fix $y\in E$ and $\varepsilon>0$. 
We take $\delta>0$ such that \eqref{supupper}. 
We can take $\bft^\delta:=\{t_k:k\in\bN\cup\{0\}\}$ such that $t_0=0$, $\lim_{m\to\infty}t_m=\infty$, $t_k-t_{k-1}\in(0,\delta)$ and $\bP_{(0,y)}(\tau^{(q)}_a=t_k)=0$ for $k\in\bN$. 
We can take $N\in\bN$ and $\bfb^\varepsilon:=\{b_k:k\in\{-1, 0,1,\dots, N\}\}$ such that $b_{-1}=b_0=0$, $b_N>\norm{f}_\infty$, $\bP_{(0, y)}(f(Y^{(q)}_{t^l})=b_k)=\bP_{(0, y)}(f(Y^{(q)}_{t^l})=b_k)=0$ and $b_k-b_{k-1}\in (0,\varepsilon)$ for $k\in\{1,\dots, N\}$ and $l\in \bN$. 
By \eqref{asdf}, Fatou's lemma, \eqref{upestemate1} and \eqref{upestemate4}, we have, for $y^\prime \in E$, 
\begin{align}
\bP_{(0, y)}(\tau^{(q)}_a<\infty)=&
\lim_{y^\prime\to y}\bP_{(0, y^\prime)}(\tau^{(q)}_a<\infty)\\
=&\liminf_{y^\prime\to y}\sum_{k\in\{0,1, \dots, M\}, m\in\bN}\bP_{(0,y^\prime)}(\Xi_a (t_{m-1}, t_{m}; b_{k-1}, b_k))
\\
&+\limsup_{y^\prime\to y}\bP_{(0,y^\prime)}(\{\tau^{(q)}_a\in\bft^\varepsilon\}\cup \bigcup_{l\in\bN}\{f(Y^{(q)}_{t^l})\in\bfb^\varepsilon\backslash\{b_0\}\})\\
&\geq\sum_{k\in\{0,1, \dots, M\}, m\in\bN}\bP_{(0,y)}(\Xi_a (t_{m-1}, t_{m}; b_{k-1}, b_k))
\\
&+\limsup_{y^\prime\to y}\bP_{(0,y^\prime)}(\{\tau^{(q)}_a\in\bft^\varepsilon\}\cup \bigcup_{l\in\bN}\{f(Y^{(q)}_{t^l})\in\bfb^\varepsilon\backslash\{b_0\}\})\\
=\bP_{(0, y)}(\tau^{(q)}_a<&\infty)+\limsup_{y^\prime\to y}\bP_{(0,y^\prime)}(\{\tau^{(q)}_a\in\bft^\varepsilon\}\cup \bigcup_{l\in\bN}\{f(Y^{(q)}_{t^l})\in\bfb^\varepsilon\backslash\{b_0\}\}),
\end{align}
and so 
\begin{align}
\limsup_{y^\prime\to y}\bP_{(0,y^\prime)}(\{\tau^{(q)}_a\in\bft^\varepsilon\}\cup \bigcup_{l\in\bN}\{f(Y^{(q)}_{t^l})\in\bfb^\varepsilon\backslash\{b_0\}\})=0. \label{allupper}
\end{align}
By \eqref{supupper}, we have, for $y^\prime \in E$,
\begin{align}
&\absol{
\bfH^{(q)}_a f(y^\prime)-\sum_{k\in\{0,1, \dots, M\}, m\in\bN}b_k\bP_{(0,y^\prime)}(\Xi_a (t_{m-1}, t_{m}; b_{k-1}, b_k))}\\
&\leq\sum_{k\in\{0,1, \dots, M\}, m\in\bN}\bE_{(0,y^\prime)}\sbra{\absol{f(Y^{(q)}_{\tau^+_a})-b_k} ;\Xi_a (t_{m-1}, t_{m}; b_{k-1}, b_k)}\\
&\qquad+\norm{f}_\infty\bP_{(0,y^\prime)}(\{\tau^{(q)}_a\in\bft^\varepsilon\}\cup \bigcup_{l\in\bN}\{f(Y^{(q)}_{t^l})\in\bfb^\varepsilon\backslash\{b_0\}\})\\
&\leq\sum_{k\in\{0,1, \dots, M\}, m\in\bN}\bE_{(0,y^\prime)}\sbra{ \bE_{(X^{(q)}_{t_{m-1}},Y^{(q)}_{t_{m-1}})}
\sbra{\absol{f(Y^{(q)}_{\tau^{(q)}_a})-b_k} 1_{\{\tau^{(q)}_a<t_m-t_{m-1}\}}};\Xi_a (t_{m-1}, \infty; b_{k-1}, b_k)}\\
&\qquad+\norm{f}_\infty\bP_{(0,y^\prime)}(\{\tau^{(q)}_a\in\bft^\varepsilon\}\cup \bigcup_{l\in\bN}\{f(Y^{(q)}_{t^l})\in\bfb^\varepsilon\backslash\{b_0\}\})\\
&\leq 2\varepsilon\bP_{(0, y^\prime)}(\tau^{(q)}_a<\infty)+\norm{f}_\infty\bP_{(0,y^\prime)}(\{\tau^{(q)}_a\in\bft^\varepsilon\}\cup \bigcup_{l\in\bN}\{f(Y^{(q)}_{t^l})\in\bfb^\varepsilon\backslash\{b_0\}\}). \label{upestemate2}
\end{align}
By Fatou's lemma, \eqref{upestemate1} and \eqref{upestemate4}, we have 
\begin{align}
\sum_{k\in\{0,1, \dots, M\}, m\in\bN}&b_k\bP_{(0,y)}(\Xi_a (t_{m-1}, t_{m}; b_{k-1}, b_k))\\
&=\sum_{k\in\{0,1, \dots, M\}, m\in\bN}b_k\liminf_{y^\prime \to y}\bP_{(0,y^\prime)}(\Xi_a (t_{m-1}, t_{m}; b_{k-1}, b_k))\\
&\leq\liminf_{y^\prime \to y}\sum_{k\in\{0,1, \dots, M\}, m\in\bN}b_k\bP_{(0,y^\prime)}(\Xi_a (t_{m-1}, t_{m}; b_{k-1}, b_k))\\
&\leq\limsup_{y^\prime \to y}\sum_{k\in\{0,1, \dots, M\}, m\in\bN}b_k\bP_{(0,y^\prime)}(\Xi_a (t_{m-1}, t_{m}; b_{k-1}, b_k))\\
&\leq\sum_{k\in\{0,1, \dots, M\}, m\in\bN}b_k\limsup_{y^\prime \to y}\bP_{(0,y^\prime)}(\Xi_a (t_{m-1}, t_{m}; b_{k-1}, b_k))\\
&=\sum_{k\in\{0,1, \dots, M\}, m\in\bN}b_k\bP_{(0,y)}(\Xi_a (t_{m-1}, t_{m}; b_{k-1}, b_k)). 
\label{upestemate5}
\end{align}
By \eqref{allupper}, \eqref{upestemate2} and \eqref{upestemate5}, we have 
\begin{align}
\limsup_{y\prime\to y}\absol{\bfH^{(q)}_a f(y^\prime)-\bfH^{(q)}_a f(y)}\leq4\varepsilon. \label{upestemate3}
\end{align}
Since \eqref{upestemate3} holds for any $y\in E$ and $\varepsilon>0$, we have $\lim_{y\prime\to y}\absol{\bfH^{(q)}_a f(y^\prime)-\bfH^{(q)}_a f(y)}=0$ for $y\in E$. 
\par
(ii) We prove 
\begin{align}
\lim_{y\to\infty}\bfH^{(q)}_a f(y)=0. 
\end{align}
We fix $\varepsilon\in(0, \norm{f}_\infty)$. Then, there exists $T\in(0, \infty)$ such that 
\begin{align}
e^{-qT}\norm{f}_\infty<\varepsilon. 
\end{align}
We have 
\begin{align}
\absol{\bfH^{(q)}_a f(y)}=& \bE_{(0, y)}\sbra{e^{-q\tau^+_a}f (Y_{\tau^+_a}); \tau^+_a\leq T}
+\bE_{(0, y)}\sbra{e^{-q\tau^+_a}f (Y_{\tau^+_a}); \tau^+_a>T}\\
\leq &\bE^Y_{ y}\sbra{\sup_{t\in[0, T]}f (Y_t)}+\varepsilon. \label{Lem501_005}
\end{align}
For $\delta>0$, we define $A_{\delta}=\{y\in E: |f(y)|>\delta\}$, which is a bounded set, and 
$\cT_{A_{\delta}}=\inf\{t>0: Y_t\in A_{\delta}\}$. 
Then, we have 
\begin{align}
\bE^Y_{y}\sbra{\sup_{t\in[0, T]}f (Y_t)} \leq \varepsilon + \bP_y(\cT_{A_{\varepsilon}}\leq T)\norm{f}_\infty. 
\label{Lem501_003}
\end{align}
We take $n\in\bN$ such that $\norm{\bfP^Y_t f-f}_\infty <\frac{\varepsilon}{2}$ for $t\in[0, \frac{T}{n}]$.  
Then, by the strong Markov property at $\cT_{A_{\varepsilon}}$, we have, for $k\in\{1, 2,\dots, n\}$, 
\begin{align}
 \bE^Y_y\sbra{f(Y_{\frac{k}{n}T}); \cT_{A_{\varepsilon}}\in (\frac{k-1}{n}T,\frac{k}{n}T]}
=&\bE^Y_y\sbra{\bfP^Y_{\frac{k}{n}T- \cT_{A_{\varepsilon}}}f (Y_{\cT_{A_{\varepsilon}}}); \cT_{A_{\varepsilon}}\in (\frac{k-1}{n}T,\frac{k}{n}T]}\\
\geq& \bP^Y_y\rbra{\cT_{A_{\varepsilon}}\in (\frac{k-1}{n}T,\frac{k}{n}T]}\rbra{\varepsilon-\frac{\varepsilon}{2}}. 
\label{Lem501_001}
\end{align}
By the property of Feller semigroups, we have 
\begin{align}
\lim_{y\to\infty} \bE^Y_y\sbra{f(Y_{\frac{k}{n}T}); \cT_{A_{\varepsilon}}\in (\frac{k-1}{n}T,\frac{k}{n}T]}=0. 
\label{Lem501_002}
\end{align}
By \eqref{Lem501_001} and \eqref{Lem501_002}, we have 
\begin{align}
\limsup_{y\to\infty}\bP_y(\cT_{A_{\varepsilon}}\leq T)\leq
\sum_{k\in\{1,\cdots,n\}} \limsup_{y\to\infty} \bP^Y_y\rbra{ \cT_{A_{\varepsilon}}\in (\frac{k-1}{n}T,\frac{k}{n}T]}
=0. 
\label{Lem501_004}
\end{align}
By \eqref{Lem501_005}, \eqref{Lem501_003} and \eqref{Lem501_004}, we have 
\begin{align}
\limsup_{y\to\infty}\absol{\bfH^{(q)}_a f(y)} \leq 2\varepsilon . 
\label{Lem501_006}
\end{align}
Since \eqref{Lem501_006} holds for any $\varepsilon\in(0, \norm{f}_\infty)$, we have $\limsup_{y\to\infty}\absol{\bfH^{(q)}_a f(y)}=0$. 

\par
(iii) We show that $\bfH^{(q)}$ is a semigroup of positive contraction operators with condition ($F_2$) in \cite[p.369]{Kal2021}. 
Positivity and the contraction property are immediate from its form. 
The condition ($F_2$) is also obtaind immediately by the dominated convergence theorem. 
So, it remains to verify the semigroup property. 
By the strong Markov peoprty and \eqref{additive_property}, we have, for $a, b\geq 0$, $f\in C_0(E)$ and $y\in E$, 
\begin{align}
\bfH^{(q)}_{a+b}f(y)=&\bE_{(0, y)}\sbra{e^{-q\tau^+_{a+b}}f (Y_{\tau^+_{a+b}}); \tau^+_{a+b}<\infty}\\
=&\bE_{(0, y)}\sbra{e^{-q\tau^+_{a}}\bE_{(X_{\tau^+_a}, Y_{\tau^+_a})}\sbra{f (Y_{\tau^+_{a+b}});\tau^+_{a+b}<\infty}; \tau^+_a<\infty}\\
=&\bE_{(0, y)}\sbra{e^{-q\tau^+_{a}}\bE_{(a, Y_{\tau^+_a})}\sbra{f (Y_{\tau^+_{a+b}});\tau^+_{a+b}<\infty}; \tau^+_a<\infty}\\
=&\bE_{(0, y)}\sbra{e^{-q\tau^+_{a}}\bE_{(0, Y_{\tau^+_a})}\sbra{f (Y_{\tau^+_{b}});\tau^+_{b}<\infty}; \tau^+_a<\infty}=\bfH^{(q)}_{a}\bfH^{(q)}_{b}f(y), 
\end{align}
and the semigroup property holds. 
The proof is complete. 
\end{proof}
\begin{Lem}\label{Lem502}
Fix $q>0$. 
The generator associated with the semigroup $\bfH^{(q)}$ is given by  
\begin{align}
 \tilde{\cA}_Y^{(q)}f(y)=\frac{1}{\ttd(y)}\rbra{\cA_Y^{(q)}f(y)+\cJ_X^{(q)}f(y)+\cJ_Y^{(q)}f(y)},\qquad y\in E, 
\end{align}
for $f\in D(\tilde{\cA}_Y^{(q)})=D(\cA_Y)$, 
where for $f\in C_0(E)$, 
\begin{align}
\cJ^{(q)}_X f (y )= 
\int_{(-\infty, 0)} \rbra{\bE_{(0, y)}\sbra{f(Y^{(q)}_{\tau^{(q)}_{-x}})}-f(y)}
\ttn(y, \diff x), \qquad y\in E,
\end{align} 
and 
\begin{align}
\cJ^{(q)}_Y f (y )= \int_E\rbra{
\int_{(-\infty, 0)} \rbra{\bE_{(0, y^\prime)}\sbra{f(Y^{(q)}_{\tau^{(q)}_{-x}})}-f(y^\prime)}
\rho_{y, y^\prime}(\diff x)} \nu(y, \diff y^\prime),\qquad y\in E. 
\end{align} 
As a consequence, $\bfH^{(q)}$ can be extended to a $C_0$-group.
\end{Lem}
\begin{proof}
Note that the generator of $\{ Y^{(q)}_t:t\geq 0\}$ is $\cA_Y^{(q)}:=\cA_Y - q\bfI$ on $D(\cA_Y)$. 
\par
(i) We take $q>0$ and $f\in D(\cA_Y)$, and define $\tilde{M}^{f,(q)}=\{\tilde{M}^{f,(q)}_a:a\geq 0\}$ by 
\begin{align}
\tilde{M}^{f,(q)}_a=f(Y^{(q)}_{\tau^{(q)}_a})-f(Y^{(q)}_0)-\int_0^a \frac{1}{\ttd (Y^{(q)}_{\tau^{(q)}_b})} 
\tilde{\cA}_Y^{(q)} f(Y^{(q)}_{\tau^{(q)}_b} )\diff b, 
\end{align}
We denote $\{\cF^{(q)}_{t}: t\geq 0\}$ the natural filtlation generated from $\{(X^{(q)}_t, Y^{(q)}_t) : t\geq 0\}$. 
We prove that $\tilde{M}^{f,(q)}$ is a martingale with respect to the filtration $\{\cF^{(q)}_{\tau^{(q)}_a}:a\geq 0\}$ under $\bP_{(0,y)}$ with $y\in E$. 
For this purpose, it suffices to show that 
\begin{align}
\bE_{(0, y)}\sbra{\tilde{M}^{f,(q)}_a}=0,\qquad a\geq 0, ~y\in E, 
\label{changedmartingale1}
\end{align}
holds. Indeed, we have, for $0\leq a <b$ and $y\in E$, 
\begin{align}
\bE_{(0 ,y)}&\sbra{\tilde{M}^{f,(q)}_b\big{|}\cF_{\tau^{(q)}_a}}\\
&=\tilde{M}^{f,(q)}_a+
\bE_{(0 ,y)}\sbra{f(Y^{(q)}_{\tau^{(q)}_b})-f(Y^{(q)}_{\tau^{(q)}_a})-\int_0^b \frac{1}{\ttd (Y^{(q)}_{\tau^{(q)}_u})} 
\tilde{\cA}_Y^{(q)} f(Y^{(q)}_{\tau^{(q)}_u} )\diff u\Bigg{|}\cF_{\tau^{(q)}_a}}\\
&=\tilde{M}^{f,(q)}_a+
\bE_{(X^{(q)}_{\tau^{(q)}_a} ,Y^{(q)}_{\tau^{(q)}_a})}\sbra{f(Y^{(q)}_{\tau^{(q)}_b})-f(Y^{(q)}_0)-\int_0^b \frac{1}{\ttd (Y^{(q)}_{\tau^{(q)}_u})} 
\tilde{\cA}_Y^{(q)} f(Y^{(q)}_{\tau^{(q)}_u} )\diff u}\\
&=\tilde{M}^{f,(q)}_a+
\bE_{(0 ,Y^{(q)}_{\tau^{(q)}_a})}\sbra{\tilde{M}^{f,(q)}_{b-a}}1_{\{\tau^{(q)}_a<\infty\}}=\tilde{M}^{f,(q)}_a,
\end{align}
where in the second equality we used the strong Markov property, in the third eaulity we used the fact that $X^{(q)}_{\tau^{(q)}_a}=a$ if $\tau^{(q)}_a<\infty$ together with \eqref{additive_property}, and the last equality we used \eqref{changedmartingale1}. 
\par 
To prove \eqref{changedmartingale1}, we time-change the martingale obtained by applying Dynkin's formula to $Y$ and compare the resulting martingale with $\tilde{M}^{f,(q)}$.
By Dynkin's formula (see, e.g., \cite[Lemma 17.21]{Kal2021}), the process $\{M^f_t:t\geq 0\}$, where 
\begin{align}
M^f_t = f(Y^{(q)}_t)-f(Y^{(q)}_0)-\int_0^t \cA_Y^{(q)} f(Y^{(q)}_s )\diff s,\qquad t\geq 0, 
\end{align}
is a martingale under $\bP_{(x, y)}$ for every $(x, y)\in\bR\times E$ with respect to the filtration $\{\cF^{(q)}_{t}: t\geq 0\}$.  
For $f \in D(\cA_Y)$ and $t>0$, we define $M^{f,(q), t}=\{M^{f,(q),t}_a:a\geq 0\}$ with $f \in D(\cA_Y)$ by 
\begin{align}
M^{f, (q), t}_a = f(Y^{(q)}_{\tau^{(q)}_a\land t})-f(Y^{(q)}_0)-\int_0^{\tau^{(q)}_a\land t} \cA_Y^{(q)} f(Y^{(q)}_s )\diff s,\qquad a\geq 0, 
\end{align}
is a martingale with respect to the filtration $\{\cF^{(q)}_{\tau^{(q)}_a}:a\geq 0\}$ by \cite[Thorem 9.12]{Kal2021}. 
Since $f$ and $\cA^{(q)}_Yf$ are bounded and by Exercises 15 in \cite[p.182]{Kal2021}, we have, for $0<a<b$ and $y\in E$, 
\begin{align}
\lim_{n\to\infty}M^{f, (q), n}_a=\lim_{n\to\infty}\bE_{(0, y)}\sbra{M^{f, (q), n}_b\big{|}\cF^{(q)}_{\tau^{(q)}_a}}
=\bE_{(0, y)}\sbra{\lim_{n\to\infty}M^{f, (q), n}_b\big{|}\cF^{(q)}_{\tau^{(q)}_a}},
\end{align} 
and $M^{f,(q)}=\{M^{f,(q)}_a:a\geq 0\}$, where 
\begin{align}
M^{f,(q)}_a=\lim_{n\to\infty} M^{f,(q), n}_a= f(Y^{(q)}_{\tau^{(q)}_a})-f(Y^{(q)}_0)-\int_0^{\tau^{(q)}_a} \cA_Y^{(q)} f(Y^{(q)}_s )\diff s,\qquad a\geq0,
\end{align}
is a martingale. 
\par 
Since $X$ has only finitely many jumps on every finite time interval and, by \eqref{L--K}, always has a strictly positive drift, we can inductively define the following stopping times and intervals: for $n\in\bN$, 
\begin{align}
&T^{X,(q)}_0= 0, \qquad  S^{X,(q)}_n=\inf\{t>T^{X,(q)}_{n-1}: X^{(q)}_{t-}=\bar{X}^{(q)}_{t-}, \Delta X^{(q)}_t \neq 0, t\not\in\bfT_\nu\}, \\
T^{X,(q)}_n&=\inf\{t>S^{X,(q)}_n: X^{(q)}_t=\bar{X}^{(q)}_t\}, \quad I^{X,(q)}_n=[S^{X,(q)}_n, T^{X,(q)}_n) , 
\quad I^{X,(q)}=\cup_{n\in\bN}I^{X,(q)}_n,
\\
&T^{Y,(q)}_0= 0, \qquad  S^{Y,(q)}_n=\inf\{t>T^{Y,(q)}_{n-1}: X^{(q)}_{t-}=\bar{X}^{(q)}_{t-}, \Delta X^{(q)}_t \neq 0, t\in\bfT_\nu\},\\
T^{Y,(q)}_n&=\inf\{t>S^{Y,(q)}_n: X^{(q)}_t=\bar{X}^{(q)}_t\}, \quad
I^{Y,(q)}_n=[S^{Y,(q)}_n, T^{Y,(q)}_n), \quad I^{Y,(q)}=\cup_{n\in\bN}I^{Y,(q)}_n,
\end{align}
where $\Delta X^{(q)}_t=X^{(q)}_t-X^{(q)}_{t-}$ and $\bar{X}^{(q)}_T=\sup_{s\in[0, t]}X^{(q)}_s$ for $t\geq 0$. 
In addition, we denote $I^{(q)}:=I^{X, (q)}\cup I^{Y,(q)}$. 
Furthermore, the objects obtained by replacing $X^{(q)}$ and $Y^{(q)}$ above with $X$ and $Y$, respectively, are denoted by removing the superscript $(q)$ from all the corresponding notation.
It follows from \eqref{L--K} that, with probability one, $X$ increases with drift rate $\ttd(Y_t)$ for $t\in[0, \infty)\backslash I^{(q)}$. Hence, 
\begin{align}
&M^{f,(q)}_a=f(Y^{(q)}_{\tau^{(q)}_a})-f(Y^{(q)}_0)-\int_0^{\tau^{(q)}_a} \cA_Y^{(q)} f(Y^{(q)}_s )\diff s\\
&=f(Y^{(q)}_{\tau^{(q)}_a})-f(Y^{(q)}_0)-\int_{[0, \tau^{(q)}_a]\backslash I^{(q)}} \cA_Y^{(q)} f(Y^{(q)}_s )\diff s
-\int_{I^{(q)}\cap[0,\tau^{(q)}_a]} \cA_Y^{(q)} f(Y^{(q)}_s )\diff s \\
&=f(Y^{(q)}_{\tau^{(q)}_a})-f(Y^{(q)}_0)-\int_0^a \frac{1}{\ttd (Y^{(q)}_{\tau^{(q)}_b})} \cA_Y^{(q)} f(Y^{(q)}_{\tau^{(q)}_b} )\diff b
-\int_{I^{(q)}\cap[0,\tau^{(q)}_a]}
\cA_Y^{(q)} f(Y^{(q)}_s )\diff s.  \label{change1}
\end{align}
By the strong Markov property and the independence of $X$, we have, for $a\geq 0$, 
\begin{align}
&\bE_{(0 ,y)}\sbra{\int_{ I^{Y,(q)}\cap[0, \tau^{(q)}_a]} \cA_Y^{(q)} f(Y^{(q)}_s )\diff s}\\
&=\sum_{n\in\bN}\bE_{(0 ,y)}\sbra{1_{\{ S^{Y,(q)}_n<\tau^{(q)}_a\}} \int_{ S^{Y,(q)}_n}^{ T^{Y,(q)}_n} \cA_Y^{(q)} f(Y^{(q)}_s )\diff s}\\
&=\sum_{n\in\bN}\bE_{(0 ,y)}\sbra{ 1_{\{ S^{Y,(q)}_n<\tau^{(q)}_a\}} \bE_{( X^{(q)}_{S^{Y,(q)}_n}, Y^{(q)}_{S^{Y,(q)}_n})}\sbra{\int_0^{ \tau^{(q)}_0} \cA_Y^{(q)} f(Y^{(q)}_s )\diff s}}\\
&=\sum_{n\in\bN}\bE_{(0 ,y)}\sbra{ 1_{\{ S^{Y,(q)}_n<\tau^{(q)}_a\}} 
\int_{(-\infty, 0)} \bE_{(x, Y^{(q)}_{S^{Y,(q)}_n})}\sbra{\int_0^{ \tau^{(q)}_0} \cA_Y^{(q)} f(Y^{(q)}_s )\diff s}
\rho_{Y_{S^{Y,(q)}_n-}, Y_{S^{Y,(q)}_n}}(\diff x)}\\
&=\bE_{(0 ,y)}\sbra{\sum_{n\in\bN} e^{-qS^{Y}_n}1_{\{ S^{Y}_n<\tau^+_a\}} 
\int_{(-\infty, 0)} \bE_{(x, Y_{S^{Y,(q)}_n})}\sbra{\int_0^{ \tau^{(q)}_0} \cA_Y^{(q)} f(Y^{(q)}_s )\diff s}
\rho_{Y_{S^{Y}_n-}, Y_{S^{Y}_n}}(\diff x)}. \label{nagaino1}
\end{align}
Since $M^{f, (q)}$ is a martingale, 
we have, for $x<0$, 
\begin{align}
\bE_{(x, y)}\sbra{\int_0^{\tau^{(q)}_0} \cA_Y^{(q)} f(Y^{(q)}_s )\diff s}
=&\bE_{(x, y)}\sbra{f(Y^{(q)}_{\tau^{(q)}_0})}-f(y). \label{mijika1}
\end{align}
By \eqref{nagaino1}, \eqref{mijika1} and Lemm \ref{LemB01}
\begin{align}
&\bE_{(0 ,y)}\sbra{\int_{ I^{Y,(q)}\cap[0, \tau^+_a]} \cA_Y^{(q)} f(Y^{(q)}_s )\diff s}\\
&=\bE_{(0 ,y)}\sbra{\sum_{n\in\bN} e^{-qS^{Y}_n}1_{\{ S^{Y}_n<\tau^+_a\}} 
\int_{(-\infty, 0)} \rbra{\bE_{(x, Y_{S^Y_n})}\sbra{f(Y^{(q)}_{\tau^+_0})}-f(Y_{S^Y_n})}
\rho_{Y_{S^{Y}_n-}, Y_{S^{Y}_n}}(\diff x)}
\\
&=\bE_{(0 ,y)}\sbra{\int_0^{\tau^+_a} e^{-qt}\cJ^{(q)}_Yf(Y_t)1_{\{X_t=\bar{X}_t\}}
\diff t}
=\bE_{(0 ,y)}\sbra{\int_0^{\tau^{(q)}_a} \cJ^{(q)}_Yf(Y^{(q)}_t)1_{\{X^{(q)}_t=\bar{X}^{(q)}_t\}}
\diff t}\\
&=\bE_{(0 ,y)}\sbra{\int_0^{a} \frac{1}{\ttd (Y^{(q)}_{\tau^{(q)}_b})}\cJ^{(q)}_Yf(Y^{(q)}_{\tau^{(q)}_b})
\diff b}. \label{change2}
\end{align}
By the strong Markov property, the independence of $X$, the compensation formula for the jump measure of $X$ and \eqref{mijika1}, we have, for $a\geq 0$, 
\begin{align}
&\bE_{(0 ,y)}\sbra{\int_{ I^{X,(q)}\cap[0, \tau^{(q)}_a]} \cA_Y^{(q)} f(Y^{(q)}_s )\diff s}\\
&=\sum_{n\in\bN}\bE_{(0 ,y)}\sbra{1_{\{ S^{X,(q)}_n<\tau^{(q)}_a\}} \int_{ S^{X,(q)}_n}^{ T^{X,(q)}_n} \cA_Y^{(q)} f(Y^{(q)}_s )\diff s}\\
&=\sum_{n\in\bN}\bE_{(0 ,y)}\sbra{ 1_{\{ S^{X,(q)}_n<\tau^{(q)}_a\}} \bE_{( X^{(q)}_{S^{X,(q)}_n}, Y^{(q)}_{S^{X,(q)}_n})}\sbra{\int_0^{ \tau^{(q)}_0} \cA_Y^{(q)} f(Y^{(q)}_s )\diff s}}\\
&=\bE_{(0 ,y)}\sbra{ \sum_{n\in\bN}e^{-q S^{X}_n} 1_{\{ S^{X}_n<\tau^+_a\}} \bE_{( X_{S^{X,(q)}_n}, Y_{S^{X,(q)}_n})}\sbra{\int_0^{ \tau^{(q)}_0} \cA_Y^{(q)} f(Y^{(q)}_s )\diff s}}\\
&=\bE_{(0 ,y)}\sbra{\int_{[0, \tau^+_a]\times (-\infty , 0)}e^{-q t}  \bE_{(x, Y_{t})}\sbra{\int_0^{ \tau^{(q)}_0} \cA_Y^{(q)} f(Y^{(q)}_s )\diff s}1_{\{X_t=\bar{X}_t\}}N_t(\diff x)}\\
&=\bE_{(0 ,y)}\sbra{\int_0^{\tau^+_a} e^{-q t}\rbra{\int_{(-\infty , 0)}  \bE_{(x, Y_{t})}\sbra{\int_0^{ \tau^{(q)}_0} \cA_Y^{(q)} f(Y^{(q)}_s )\diff s}\ttn(Y_t, \diff x)}1_{\{X_t=\bar{X}_t\}}\diff t}\\
&=\bE_{(0 ,y)}\sbra{\int_0^{\tau^+_a} e^{-q t}\cJ_X^{(q)}f(Y_t)1_{\{X_t=\bar{X}_t\}}\diff t}
=\bE_{(0 ,y)}\sbra{\int_0^{\tau^{(q)}_a} \cJ_X^{(q)}f(Y^{(q)}_t)1_{\{X^{(q)}_t=\bar{X}^{(q)}_t\}}\diff t}
\\
&=\bE_{(0 ,y)}\sbra{\int_0^{a} \frac{1}{\ttd (Y^{(q)}_{\tau^{(q)}_b})}\cJ^{(q)}_Xf(Y^{(q)}_{\tau^{(q)}_b})
\diff b}. \label{change3}
\end{align}
By \eqref{change1}, \eqref{change2} and \eqref{change3}, we have
\begin{align}
\bE_{(0, y)}\sbra{M^{f,(q)}_a}=\bE_{(0, y)}\sbra{\tilde{M}^{f,(q)}_a},\qquad a\geq 0, y\in E, 
\label{changedmartingale}
\end{align}
which implies \eqref{changedmartingale}.  
\par
(ii) 
We show that the generator of the semigroup $\bfH^{(q)}$ is $ \tilde{\cA}_Y^{(q)}$ with domain $D(\tilde{\cA}_Y^{(q)})=D(\cA_Y)$, and that $\bfH^{(q)}$ extends to a \(C_0\)-group. 
By \eqref{finiteness1} and \eqref{finiteness2}, it is easy to confirm that $\cJ_X$ and $\cJ_Y$ are bounded operator on $C_0(E)$. 
By \cite[p.79]{EngNag2000}, the main theorem in \cite{Dor1966} and its proof, the operator $\frac{1}{\ttd}\cA_{Y^\prime}$ defined by 
\begin{align}
\frac{1}{\ttd}\cA_{Y^\prime}f(y)=\frac{1}{\ttd (y)}\cA_{Y^\prime}f(y),\qquad y\in E,~f\in C_0(E), 
\end{align}
is also a generator of a $C_0$-group, with domain $D(\cA_Y)$. 
Since $\tilde{\cA}_Y^{(q)}-\frac{1}{\ttd}\cA_{Y^\prime}$ is a bounded linear operator, it follows from \cite[p.79 and Theorem III.1.3]{EngNag2000} that $\tilde{\cA}_Y^{(q)}$ also generates a $C_0$-group and has domain $D(\cA_Y)$. 
The proof is complete. 
\end{proof}
\subsection{The definition of the scale operators and the proofs of \eqref{scale_Laplace_transform}}\label{subsec402}
In the main theorems, the scale operators are essentially defined in terms of Theorems \ref{Thm300} (iii).
In this subsection, however, we define and characterize scale operators using $\{\bfH^{(q)}_x : x\geq 0\}$, local times, and exit systems, rather than defining it in terms of before turning to Theorem \ref{Thm300} (iii), and prove the main results other than Theorems \ref{Thm300} (iii). 
\par 
Before proceeding to the proof, we would like to keep the following remark in mind.
\begin{Rem}\label{BassSection6}
From now on, we use $T_n^X$ and $T_n^Y$ to denote the times of the $n$-th jumps of $X$ and $Y$, respectively, 
where $T^0_X=T^0_Y=0$.
By \cite[Section 6]{Bas1979}, the process $(X, Y)$ can be constructed by concatenating continuous path segments and jumps. 
We first review the construction of $Y$. 
For $n\in\bN$, 
$T_1^Y- T_0^Y$ can be represented as $\inf\{t>0: \int_0^t\nu( T_s, E)\diff s \geq  {\bf e}_1\}$ for some independent exponential rabdom variable ${\bf e}_1$ with intensity $1$ 
and can be bounded by exponential distributions with intensity $M_\nu$, and 
the distribution of the jump size at time $t\in(0,\infty)$ can be regarded as 
\begin{align}
\tilde{\nu}(Y_{t-}, \cdot):=\frac{\nu(Y_{t-}, \cdot)}{\nu(Y_{t-}, E)}, 
\end{align}
where $\frac{0}{0}=0$. 
Similarly, $Y$ can be constructed by concatenating continuous drift segments and jumps. 
Since $X$ behaves as an additive process with the characteristic exponent \eqref{L--K}, $X$ has two types of jumps. 
One type of jump occurs simultaneously with a jump of $Y$ 
and the distribution of the jump size at time $t\in(0,\infty)$ can be regarded as 
\begin{align}
\tilde{\rho}(Y_{t-}, \cdot):=\int_E \rho_{Y_{t-}, y^\prime}(\cdot) 
\tilde{\nu}(Y_{t-}, \diff y^\prime)
\end{align}
For the other type, the inter-jump times are bounded below by exponential random variables with rate $M_\ttn$, and the distribution of the jump size at time $t\in(0,\infty)$ can be regarded as 
\begin{align}
\tilde{\ttn}(Y_t,\cdot):=
\frac{\ttn(Y_t, \cdot)}{\ttn(Y_t, (-\infty ,0))}. 
\end{align} 
Consequently, the inter-jump times of $X$ are bounded below by exponential random variables with rate $M_\nu+M_\ttn$.
\end{Rem}
\par
For $q>0$ and $x \geq 0$, we define the operator $\ttE^{(q)}_x$, $\bfL^{(q)}_x$, $\bfW^{(q)}_x$ on $C_0(E)$ by, for $ f\in C_0 (E)$,
\begin{align}
\ttE^{(q)}_xf(y):=& \cE_y^0\sbra{e^{-q\tau^+_x}f(Y_{\tau^+_x}); \tau^+_x<\infty}=\ttd(y)\bfQ^{(q),[0,x]}_0f(y), \\
\bfL^{(q)}_xf(y):= &\bE_{(0,y)}\sbra{\int_{[0, \tau^+_x]}e^{-qs} f(Y_s)\diff L^0_s}. 
\end{align}
\begin{Lem}\label{Lem403}
For $q>0$, $x \geq 0$ and $f\in C_0(E)$, $\ttE^{(q)}_xf \in C_0(E)$. 
\end{Lem}
The proof of Lemma \ref{Lem403}, although more involved, follows essentially the same arguments as those in parts (i) and (ii) of the proof of Lemma \ref{Lem401} and is therefore omitted. 
\begin{Lem}\label{Lem404}
For $q>0$, $x \geq 0$ and $f\in C_0(E)$, the operator $\bfL^{(q)}_x$ is injective and $\bfL^{(q)}_xf\in C_0(E)$. 
\end{Lem}
\begin{proof}
We fix $q>0$ and $x\geq 0$, and assume that the operator $\bfL^{(q)}_x$ is not injective. 
Then, we can take $f\in C_0(E)$ such that $f\not\equiv 0$ and $\bfL^{(q)}_x f\equiv 0$. We can also take $y\in E$ such that $f(y)\neq 0$. 
Since $\bfL^{(q)}_x f\equiv 0$ and by the strong Markov property at $T_0$, we have 
\begin{align}
0=\bfL^{(q)}_x f(y)-\bE_{(0, y)}\sbra{e^{-q T_0} \bfL^{(q)}_x f(Y_{T_0}); T_0\leq \tau^+_x}
=\bE_{(0,y)}\sbra{e^{-qs} f(Y_0) L^0_0}=f(y)\neq 0. 
\end{align}
This contradiction shows that no $f$ satisfies the above conditions, and hence $\bfL^{(q)}_x$ is injective. 
\par
We prove that $\bfL^{(q)}_x f \in C_0(E)$ for $f\in C_0 (E)$. 
Since the jump size of $L^0$ is less than $\bar{d}$ and the process must make at least one jump to return to $0$ after starting from $0$ and by Remark \ref{BassSection6}, which gave the bound of jump intervals, we have 
\begin{align}
\norm{\bfL^{(q)}_xf}_\infty &\leq \bar{d}\norm{f}_\infty\rbra{1+\sum_{n\in\bN}\rbra{(M_\nu+M_\ttn)\int_0^\infty e^{-(q+M_\nu+ M_\ttn)t}\diff t}^n}\\
&=\bar{d}\norm{f}_\infty\frac{q+M_\nu+M_\ttn}{q}. \label{Lnormupper}
\end{align}
By \eqref{Lnormupper} and the strong Markov property at $T^{\cT_{A_\varepsilon}}_0:=\inf\{t> \cT_{A_\varepsilon}: X_t=0\}$ with $\varepsilon>0$, 
we have 
\begin{align}
\lim_{y\to\infty}\absol{\bfL^{(q)}_xf(y)}
\leq&\lim_{y\to\infty}\bE_{(0,y)}\sbra{\int_{[0, \tau^+_x\land \cT_{A_\varepsilon}]}e^{-qs} |f(Y_s)|\diff L^0_s}\\
&+\lim_{y\to\infty}\bE_{(0,y)}\sbra{e^{-qT^{\cT_{A_\varepsilon}}_0}\absol{\bfL^{(q)}_xf(Y_{T^{\cT_{A_\varepsilon}}_0})} }\\
\leq& \bar{d}\varepsilon\frac{q+M_\nu+M_\ttn}{q}+\lim_{y\to\infty} \bE_{(0,y)}\sbra{e^{-qT^{\cT_{A_\varepsilon}}_0}}\bar{d}\norm{f}_\infty\frac{q+M_\nu+M_\ttn}{q}\\
\leq&\bar{d}\varepsilon\frac{q+M_\nu+M_\ttn}{q}. \label{Lnormupper2}
\end{align}
Since \eqref{Lnormupper2} holds for any $\varepsilon>0$, we have $\lim_{y\to\infty}\absol{\bfL^{(q)}_xf(y)}=0$. 
It remains to show that $\bfL^{(q)}_xf$ is continuous. 
We assume that there exists a non-negative function $f\in C_0 (E)$ such that $\bfL^{(q)}_{x}f$ is not continuous at $\tilde{y}\in E$. 
Then, there exists $\gamma>0$ such that for any $n\in\bN$, we can take $y_n\in E$ such that 
$d(\tilde{y} , y_n)< \frac{1}{n}$ and $\absol{\bfL^{(q)}_{x}f(\tilde{y})-\bfL^{(q)}_{x}f(y_n)}>\gamma$. 
In particular, we may assume without loss of generality that $\bfL^{(q)}_{x}f(y_n)-\bfL^{(q)}_{x}f(\tilde{y})>\gamma$ holds for $n\in\bN$.
For $\varepsilon >0$, we define 
\begin{align}
\bfL^{(q)}_{x, \varepsilon}f(y):= \bE_{(0,y)}\sbra{\int_{[0, \tau^+_x]}e^{-qs} f(Y_s)\diff L^\varepsilon_s}. 
\end{align}
By Assumption \ref{Ass203} (ii), the time $X$ needed to hit $\varepsilon>0$ after hitting $0$ is at most $\frac{\varepsilon}{\un{d}}$ provided that no jump occurs. 
Moreover, by Remark \ref{BassSection6}, the probability of no jump during this interval is at least
$e^{-\frac{\varepsilon}{\underline d}(M_\nu+M_{\mathtt n})}$.
For $\varepsilon>0$, we can take $\delta>0$ which satisfies \eqref{supupper} for $Y^\prime$. 
Furthermore, taking into account the argument in \eqref{Lnormupper}, we have, 
for $a \in (0, (\delta\land\varepsilon)\bar{d})$ and $y\in E$, 
\begin{align}
\bfL^{(q)}_{x}f(y)-\bfL^{(q)}_{x, a}f(y)
\leq &\bar{d}\rbra{e^{-\frac{\varepsilon}{\un{d}}(M_\nu+M_{\mathtt n})}\varepsilon  +(1-e^{-\frac{\varepsilon}{\un{d}}(M_\nu+M_{\mathtt n})})\norm{f}_\infty}\\
&\times\rbra{1+\sum_{n\in\bN}\rbra{(M_\nu+M_\ttn)\int_0^\infty e^{-(q+M_\nu+ M_\ttn)t}\diff t}^n}\\
= &\bar{d}\rbra{e^{-\frac{\varepsilon}{\un{d}}(M_\nu+M_{\mathtt n})}\varepsilon  +(1-e^{-\frac{\varepsilon}{\un{d}}(M_\nu+M_{\mathtt n})})\norm{f}_\infty}\frac{q+M_\nu+M_\ttn}{q}.
\label{Lconbound}
\end{align}
Moreover, we can choose $c>0$ such that, with
$\rho_c:=\inf\{t>0:\Delta X_t\neq 0,\ X_t\in(0,c]\}$,
we have $\bE_{(0,\tilde{y})}\sbra{\exp(-q\rho_c)}<\varepsilon$. Hence, combining \eqref{Lconbound} with the strong Markov property at $\rho_c$ and an estimate analogous to that used in the proof of \eqref{Lnormupper2}, we have, for $a \in (0, \cbra{(\delta\land\varepsilon)\bar{d}}\land c)$ and $y\in E$, 
\begin{align}
\absol{\bfL^{(q)}_{x}f(\tilde{y})-\bfL^{(q)}_{x, a}f(\tilde{y})}
\leq &\bar{d}\rbra{e^{-\frac{\varepsilon}{\un{d}}(M_\nu+M_{\mathtt n})}\varepsilon  +(1-e^{-\frac{\varepsilon}{\un{d}}(M_\nu+M_{\mathtt n})}+\varepsilon)\norm{f}_\infty}\frac{q+M_\nu+M_\ttn}{q}
\label{Lconbound3}
\end{align}
Now choose $\varepsilon>0$ so that the right-hand side of \eqref{Lconbound} and \eqref{Lconbound3} are less than $\frac{\gamma}{3}$. Then there exists $b>0$ such that, for every $a\in[0,b]$, the point $y^\prime$ used in the earlier argument satisfies
\begin{align}
\bfL^{(q)}_{x, a}f(y_n)-\bfL^{(q)}_{x, a}f(y)>\frac{\gamma}{3}, \qquad n\in\bN. 
\end{align}
The above fact contradicts the continuity at $\tilde{y}$ of 
\begin{align}
y\mapsto
&\bE_{(0 ,y)}\sbra{\int_0^\infty e^{-qt}g(X_t)f(Y_t)\diff t}
-\bE_{(0, y)}\sbra{ e^{-q\tau^+_x}\bE_{(x ,Y_{\tau^+_x})}\sbra{\int_0^\infty e^{-qt}g(X_t)f(Y_t)\diff t};\tau^+_x<\infty}\\
&=\bE_{(0 ,y)}\sbra{\int_0^{\tau^+_x} e^{-qt}g(X_t)f(Y_t)\diff t}=\int_{[0,\infty)}g(x)\bfL^{(q)}_{x, \varepsilon}f(y) \diff x, 
\end{align}
where $g\in C_0(E)$ is non-negative and not identically zero, with support contained in $[0,\cbra{(\delta\land\varepsilon)\bar{d}}\land c]$.
The proof is complete. 
\end{proof}
\par
The following lemma was proved in \cite[Lemma 3.5]{Nob2020a} in the one-dimensional setting, while in the joint work in progress with Professors Jos\'e Luis P\'erez and V\'ictor Rivero, the case where $Y$ is a Markov chain was established.
\begin{Lem}\label{Lem503}
For $q>0$ and $x\geq 0$, we have $\bfH^{(q)}_x=\bfL^{(q)}_x\ttE^{(q)}_x$. 
\end{Lem}
\begin{proof}
By Lemma \ref{Lem302}, we have, for $f\in C_0(E)$, $x\geq 0$ and $y\in E$, 
\begin{align}
\bE_{(0,y)}\sbra{e^{-q\tau^+_x}f(Y_{\tau^+_x});\tau^+_x<\infty}
&=\bE_{(0,y)}\sbra{\sum_{t \in G^0}1_{\{t\leq\tau^+_x\}}  e^{-qt}\Big{(}e^{-q\tau^+_x}f(Y_{\tau^+_x})1_{\{\tau^+_x<\infty\}}\circ\theta_t\Big{)}}\\
&=\bE_{(0,y)}\sbra{\int_{[0, \tau^+_x]}e^{-qs} \cE^0_{Y_s}\sbra{e^{-q\tau^+_x}f(Y_{\tau^+_x});\tau^+_x<\infty}\diff L^0_s}, \label{split_formula}
\end{align}
so we have $\bfH^{(q)}_x=\bfL^{(q)}_x\ttE^{(q)}_x$. 
The proof is complete. 
\end{proof}
From Lemma \ref{Lem503}, 
$\bfL^{(q)}_xf \in C_0(E)$ for $q>0$, $x \geq 0$ and $f\in C_0(E)$. 
Thus, combining this with Lemmas \ref{Lem502}, \ref{Lem403} and \ref{Lem404}, we can define, for $q>0$ and $x \geq 0$, 
\begin{align}
\bfW^{(q)}_xf(y)= \bfH^{(q),-1}_x\bfL^{(q)}_x f(y),\qquad f\in C_0(E), 
\end{align}
and we have 
\begin{align}
 \bfW^{(q), -1}_x =\ttE^{(q)}_x ,  \label{scaleinverse}
\end{align}
and Theorem \ref{Thm300} (i) holds. 
\par

We present the following lemma for Theorem \ref{Thm300} (ii). 
\begin{Lem}\label{Lem504}
For $q>0$ and $f\in C_0(E)$, the map $(x, y)\mapsto \mathbf W_x^{(q)}f(y)$ is $\cB(\bR)$-measurable on $[0,\infty)$. 
In addition, for $y\in E$, $x\mapsto \mathbf W_x^{(q)}f(y)$ is non-decreasing. 
\end{Lem}
\begin{proof}
By Theorem \ref{Thm401} and \eqref{additive_property}, we have, for $q>0$, $f\in C_0(E)$, $x, b\in\bR$ with $0\leq x\leq b$ and $y\in E$, 
\begin{align}
\bfW^{(q)}_xf(y)&=\bfW^{(q)}_x{\bfW^{(q),-1}_b} \bfW^{(q),-1}_bf(y) \\
&=\bE_{(x, y)}\sbra{e^{-q\tau^+_b} \bfW^{(q),-1}_bf(Y_{\tau^+_b});\tau^+_b<\tau^-_0}. 
\end{align}
This function is non-decreasing in $x\in[0,b]$, and since $x\leq b$ are arbitrary, it is non-decreasing on $[0,\infty)$.
Since $e^{-q\tau^+_b} \bfW^{(q),-1}_bf(Y_{\tau^+_b})1_{\{\tau^+_b<\tau^-_0\}}$ is $\cup_{t\geq 0}\cF_t$-measurable, the above function is $\cB(\bR)$-measurable. 
The proof is complete. 
\end{proof}

The proofs of Theorems \ref{Thm401} and \ref{Thm402} are based on the arguments used in the proofs of \cite[Theorem 3.4, Lemma 3.4, and Theorem 3.6]{Nob2020a}. 
\begin{proof}[Proof of Theorem \ref{Thm401}]
By the strong Markov property of $\cE^a$, \eqref{additive_property} and since $X_{\tau^+_x}=x$ on $\{\tau^+_x<\infty\}$ under $\cE^a$ when $a\leq x$, we have, for $q>0$, $a,b,x\in\bR$ with $a\leq x\leq b$ and $f\in C_0(E)$,
\begin{align}
\cE^0_y\sbra{e^{-q\tau^+_{b-a}}f(Y_{\tau^+_{b-a}}); \tau^+_{b-a}<\infty}
&=\cE^a_y\sbra{e^{-q\tau^+_b}f(Y_{\tau^+_b}); \tau^+_b<\infty}\\
&=\cE^a_y\sbra{e^{-q\tau^+_x}\bE_{(x, Y_{\tau^+_x})}\sbra{f(Y_{\tau^+_b});\tau^+_b<\tau^-_a};\tau^+_x<\infty}\\
&=\cE^0_y\sbra{e^{-q\tau^+_{x-a}}\bE_{(x, Y_{\tau^+_{x-a}})}\sbra{f(Y_{\tau^+_b});\tau^+_b<\tau^-_a};\tau^+_{x-a}<\infty}
\end{align}
and hence 
we obtain $\ttE^{(q)}_{b-a}=\ttE^{(q)}_{x-a}\bfQ^{(q),[a,b]}_x$. 
Applying $\bfW^{(q)}_{x-a}$ to both sides of this identity and 
\eqref{scaleinverse}, the proof is complete. 
\end{proof}
\begin{proof}[Proof of Theorem \ref{Thm402}]
We fix $q>0$, $a\leq 0\leq b$ and $f\in C_0(E)$. 
By the strong Markov property at the jumptimes of $L^0$, 
the same computation as \eqref{split_formula} and Theorem \ref{Thm401}, we have  
\begin{align}
\bE_{(0,y)}&\sbra{\int_{[0, \tau^-_a\land \tau^+_b]}e^{-qs} f(Y_s)\diff L^0_s}\\
&=\bE_{(0,y)}\sbra{\int_{[0, \tau^-_a\land \tau^+_b]}e^{-qs} \cE^0_{Y_s}\sbra{e^{-q\tau^+_b} \bfW^{(q)}_bf(Y_{\tau^+_b});\tau^+_b<\infty}\diff L^0_s}\\
&=\bE_{(0,y)}\sbra{e^{-q\tau^+_b} \bfW^{(q)}_bf(Y_{\tau^+_x});\tau^+_b<\tau^-_a}\\
&= \bfW^{(q)}_{-a}\bfW^{(q),-1}_{b-a} \bfW^{(q)}_bf(y). \label{Keypotential}
\end{align}
For $x\in(0,b]$, we have 
\begin{align}
\bfL^{(q),[a,b]}_{\{x\}}f(y)&=\bE_{(0,y)}\sbra{e^{-q\tau^+_x} \bE_{(X_{\tau^+_x}, Y_{\tau^+_x})}\sbra{\int_{[0, \tau^-_a\land \tau^+_b]}e^{-qs} f(Y_s)\diff L^x_s};\tau^+_x<\tau^-_a}\\
&=\bE_{(0,y)}\sbra{e^{-q\tau^+_x} \bE_{(0, Y_{\tau^+_x})}\sbra{\int_{[0, \tau^-_{a-x}\land \tau^+_{b-x}]}e^{-qs} f(Y_s)\diff L^0_s};\tau^+_x<\tau^-_a}\\
&=\bfW^{(q)}_{-a}\bfW^{(q),-1}_{x-a}\bfW^{(q)}_{x-a}\bfW^{(q),-1}_{b-a} \bfW^{(q)}_{b-x}f(y)\\
&=\bfW^{(q)}_{-a}\bfW^{(q),-1}_{b-a} \bfW^{(q)}_{b-x}f(y),
\end{align}
where in the first equality we used the strong Markov property, in the second equality we used the fact that $X_{\tau^+_x}=x$ on $\{\tau^+_x<\infty\}$ and \eqref{additive_property} and in the third equality we used Theorem \ref{Thm402} and \eqref{Keypotential}. 
We define $T_x=\inf\{t>0: X_t=x\}$ for $x\in\bR$. 
By the strong Markov property and \eqref{additive_property}, we have, for $x\in[a, 0]$,  
\begin{align}
\bfL^{(q),[a,b]}_{\{x\}}f(y)&=\bE_{(0,y)}\sbra{e^{-qT_x} \bE_{(x, Y_{T_x})}\sbra{\int_{[0, \tau^-_a\land \tau^+_b]}e^{-qs} f(Y_s)\diff L^x_s};T_x<\tau^-_a\land \tau^+_b}\\
&=\bE_{(0,y)}\sbra{e^{-qT_x} \bE_{(0, Y_{T_x})}\sbra{\int_{[0, \tau^-_{a-x}\land \tau^+_{b-x}]}e^{-qs} f(Y_s)\diff L^x_s};T_x<\tau^-_a\land \tau^+_b}
. \label{potentialexplicit1}
\end{align}
On the other hand, by the strong Markov property, we have, for $x\in[a, 0]$ and $f\in C_0(E)$, 
\begin{align}
\bfQ^{(q),[a,b]}_0f(y)-\bfQ^{(q),[x,b]}_0f(y)
&=\bE_{(0, y)}\sbra{e^{-q\tau^+_b}f(Y_{\tau^+_b});T_x<\tau^+_b<\tau^-_a}\\
&=\bE_{(0,y)}\sbra{e^{-qT_x} \bfQ^{(q),[a,b]}_x f(Y_{T_x});T_x<\tau^-_a\land \tau^+_b}. 
\label{potentialexplicit2}
\end{align}
By \eqref{potentialexplicit1} and \eqref{potentialexplicit2}, we have, for $x\in[a, 0]$,  
\begin{align}
\bfL^{(q),[a,b]}_{\{x\}}f(y)=\bfQ^{(q),[a,b]}_0\bfQ^{(q),[a,b],-1}_xh_f(y)-\bfQ^{(q),[x,b]}_0\bfQ^{(q),[a,b],-1}_xh_f(y),
\end{align}
where 
\begin{align}
h_f(y)=\bE_{(0,y)}\sbra{\int_{[0, \tau^-_{a-x}\land \tau^+_{b-x}]}e^{-qs} f(Y_s)\diff L^0_s},\qquad y\in E. 
\end{align}
Combining this with Theorem \ref{Thm402} and \eqref{Keypotential}, we have
\begin{align}
\bfL^{(q),[a,b]}_{\{x\}}f(y)=&\bfW^{(q)}_{-a}\bfW^{(q),-1}_{b-a}\bfW^{(q)}_{b-a}\bfW^{(q),-1}_{x-a} 
\bfW^{(q)}_{x-a}\bfW^{(q),-1}_{b-a} \bfW^{(q)}_{b-x}f(y)\\
&-\bfW^{(q)}_{-x}\bfW^{(q),-1}_{b-x}\bfW^{(q)}_{b-a}\bfW^{(q),-1}_{x-a} 
\bfW^{(q)}_{x-a}\bfW^{(q),-1}_{b-a} \bfW^{(q)}_{b-x}f(y)\\
=&\bfW^{(q)}_{-a}\bfW^{(q),-1}_{b-a} \bfW^{(q)}_{b-x}f(y)
-\bfW^{(q)}_{-x}f(y).
\end{align}
The proof is complete
\end{proof}

\subsection{The proof of Theorem \ref{Thm300} (iii)}
In this subsection, we show that the scale operators, defined in Subsection \ref{subsec402}, satisfies the conditions stated in Theorem \ref{Thm300} (iii). 
In the case of a spectrally negative L\'evy processes, we can naturally characterize the Laplace transform of the scale function by taking the Laplace transform of \cite[(12)]{BifKyp2010} that comes out of the two-sided exit problem. 
We will do the proof with the same idea, but since we do not go through the Meyer--It\^o formula, the calculations get complicated. 
\par
For $\beta\geq 0$, $f\in C_0(E)$ and $y\in E$, we define 
\begin{align}
\bff^{(\beta)}_1(y)&:=\beta \ttd(y)+\int_{(-\infty , 0)} \rbra{e^{\beta x}-1 } \bfn (y, \diff x), 
\\
\bfF^{(\beta)}_2f(y)&:=\int_{(-\infty, 0)}\rbra{F_{y,y^\prime}(\beta)f(y^\prime)-f(y)}\nu(y, \diff y^\prime).
\end{align}
Note that $\beta\geq 0$ and $f\in D(\cA_Y)$, we have 
\begin{align}
\bfF^{(\beta)} f(y)=\bff^{(\beta)}_1(y) f(y)+\bfF^{(\beta)}_2f(y) +\cA_{Y^\prime} f(y),\qquad y\in E. \label{Dbetadecom}
\end{align}

\begin{Lem}\label{Lem406}
For $\beta\geq 0$ and $f\in C_0(E)$, we have 
\begin{align}
\lim_{t\downarrow0}\sup_{y\in E}
\absol{\frac{\bE_{(0,y)} \sbra{e^{\beta X_t}f( Y_t) }-f(y)}{t}-\frac{\bE_{(0,y)}\sbra{f(Y_t)-f(y); t<T^Y_1}}{t}-\bff^{(\beta)}_1(y)f(y)-\bfF^{(\beta)}_2f(y)}=0 . 
\end{align}
\end{Lem}
\begin{proof}
We decompose as 
\begin{align}
&\frac{\bE_{(0,y)} \sbra{e^{\beta X_t}f( Y_t) }-f(y)}{t}=
\frac{\bE_{(0,y)} \sbra{\bE_0^{\omega}[e^{\beta X_t}]f( Y_t) }-f(y)}{t}
\\
&=\frac{1}{t}\bE_{(0,y)}\Bigg{[}
J^Y_t(\beta)
\rbra{\exp\cbra{ \int_0^t \bff^{(\beta)}_1(Y_s) \diff s}
-1}f(Y_t)\Bigg{]}
\\
&\qquad+\frac{\bE_{(0,y)}\sbra{J^Y_t(\beta)f(Y_t)}-f(y)}{t},
\label{40602}
\end{align}
where $J^Y_t(\beta):=\prod_{s\in[0,t]}F_{Y_{s-},Y_s}(\beta)$ for $t\geq 0$ and $\beta\geq0$. 
For $f\in C_0(E)$ and $\beta \geq 0$, we have 
\begin{align}
&\lim_{t\downarrow0}\sup_{y\in E}\absol{\frac{1}{t}\bE_{(0,y)}
\sbra{J^Y_t(\beta)
\rbra{\exp\cbra{ \int_0^t\bff^{(\beta)}_1(Y_s)\diff s}
-1}f(Y_t)}-\bff^{(\beta)}_1(y)f (y)}\\
&\leq\lim_{t\downarrow0} \sup_{y\in E}\absol{\bE_{(0,y)}
\sbra{\frac{1}{t}J^Y_t(\beta)
\rbra{\exp\cbra{ \int_0^t\bff^{(\beta)}_1(Y_s)\diff s}
-1}f(Y_t)-\bff^{(\beta)}_1(y)f (y); T^Y_1>t}}\\
&+\lim_{t\downarrow0} \sup_{y\in E}\absol{\bE_{(0,y)}
\sbra{\frac{1}{t}J^Y_t(\beta)
\rbra{\exp\cbra{ \int_0^t\bff^{(\beta)}_1(Y_s)\diff s}
-1}f(Y_t)-\bff^{(\beta)}_1(y)f (y); T^Y_1\leq t}}\\
&\leq\lim_{t\downarrow0} \sup_{y\in E}  
\absol{\bE_{(0,y)}\sbra{\frac{1}{t}\rbra{\int_0^t\bff^{(\beta)}_1(Y^\prime_u)\exp\cbra{ \int_0^u\bff^{(\beta)}_1(Y^\prime_s)\diff s}\diff u}f(Y^\prime_t)-\bff^{(\beta)}_1(y)f (y)}}\\
&+\lim_{t\downarrow0} (1-e^{-M_\nu t})\cdot 2(\beta\bar{d}+M_\nu)e^{t(\beta+M_\nu)} \norm{f}_\infty \\
&
\leq\lim_{t\downarrow0} \sup_{y\in E}  
\bE_{(0,y)}\sbra{\sup_{u\in[0, t]}\absol{\bff^{(\beta)}_1(Y^\prime_u)\exp\cbra{ \int_0^u\bff^{(\beta)}_1(Y^\prime_s)\diff s}f(Y^\prime_t)-\bff^{(\beta)}_1(y)f (y)}}\\
&\leq\lim_{t\downarrow0} \sup_{y\in E}  
\bE_{(0,y)}\sbra{\sup_{u\in[0, t]}\absol{\bff^{(\beta)}_1(Y^\prime_u)\exp\cbra{ \int_0^u\bff^{(\beta)}_1(Y^\prime_s)\diff s}f(Y^\prime_u)-\bff^{(\beta)}_1(y)f (y)}}\\
&+\lim_{t\downarrow0} \sup_{y\in E}  
\bE_{(0,y)}\sbra{\sup_{u\in[0, t]}\absol{\bff^{(\beta)}_1(Y^\prime_u)\exp\cbra{ \int_0^u\bff^{(\beta)}_1(Y^\prime_s)\diff s}(f(Y^\prime_t)-f(Y^\prime_u))}}
=0, 
\label{40603}
\end{align}
where in the second inequality we used Remark \ref{BassSection6} and Assumption \ref{Ass203} (ii) and (iii) 
and in the last equality we used Assumption \ref{Ass203} (ii) and (iii) and \eqref{supupper} for $Y^\prime$. 
We can sprit
\begin{align}
&\frac{\bE_{(0,y)}\sbra{J^Y_t(\beta)f(Y_t)}-f(y)}{t}\\
&=\frac{\bE_{(0,y)}\sbra{f(Y_t)-f(y); t<T^Y_1}}{t}+
\frac{\bE_{(0,y)}\sbra{J^Y_t(\beta)f(Y_t)-f(y);T^Y_1\leq t<T^Y_2}}{t}\\
&\quad +\frac{\bE_{(0,y)}\sbra{J^Y_t(\beta)f(Y_t)-f(y);T^Y_2\geq t}}{t}. 
\label{40601new}
\end{align}
By Remark \ref{BassSection6}, we have, for $f\in C_0(E)$ and $\beta \geq 0$, 
\begin{align}
\lim_{t\downarrow0}\sup_{y\in E}&\absol{\frac{\bE_{(0,y)}\sbra{J^Y_t(\beta)f(Y_t)-f(y);T^Y_2\geq t}}{t}}\\
&\qquad\qquad=\lim_{t\downarrow0}2\norm{f}_\infty\absol{\frac{1-(1+M_\nu t)e^{-M_\nu t}}{t}}=0. \label{40604}
\end{align} 
We have, for $f\in C_0(E)$ and $\beta \geq 0$
\begin{align}
&\lim_{t\downarrow0}\sup_{y\in E} \absol{\frac{\bE_{(0,y)}\sbra{J^Y_t(\beta)f(Y_t)-f(y);T^Y_1\leq t<T^Y_2}}{t}
-\bfF^{(\beta)}_2f(y)}\\
&=\lim_{t\downarrow0}\sup_{y\in E} \bigg{|}\frac{1}{t}
\int_{(0, t)}\bP_{(0,y)}\rbra{T^Y_1 \in \diff s}\\
&\qquad\int_{(-\infty, 0)}\rbra{F_{\phi_s(y),y^\prime}(\beta)f(\phi_{y^\prime}(t-s))-f(\phi_s(y))}
\tilde{\nu}(\phi_s(y), \diff y^\prime)
-\bfF^{(\beta)}_2f(y)\bigg{|}\\
&=\lim_{t\downarrow0}\sup_{y\in E} \bigg{|}\frac{1}{t}
\int_{(0, t)}\nu(\phi_s(y), (-\infty,0)) \exp\cbra{\int_0^s\nu(\phi_u(y), (-\infty,0))\diff u}\diff s\\
&\qquad\int_{(-\infty, 0)}\rbra{F_{\phi_s(y),y^\prime}(\beta)f(\phi_{y^\prime}(t-s))-f(\phi_s(y))}
\tilde{\nu}(\phi_s(y), \diff y^\prime)
-\bfF^{(\beta)}_2f(y)\bigg{|}\\
&\leq\lim_{t\downarrow0}\sup_{y\in E} \absol{\frac{1}{t}
\int_{(0, t)} \exp\cbra{\int_0^s\nu(\phi_u(y), (-\infty,0))\diff u}\bfF^{(\beta)}_2f(\phi_s(y))\diff s
-\bfF^{(\beta)}_2f(y)}\\
&\quad+\lim_{t\downarrow0}\sup_{y\in E} \bigg{|}\frac{1}{t}
\int_{(0, t)} \exp\cbra{\int_0^s\nu(\phi_u(y), (-\infty,0))\diff u}\diff s\\
&\quad\qquad\int_{(-\infty, 0)}F_{\phi_s(y),y^\prime}(\beta)\rbra{f(\phi_{y^\prime}(t-s))-f(y^\prime)}
\nu(\phi_s(y), \diff y^\prime)\bigg{|}\\
&\leq\lim_{t\downarrow0}\sup_{y\in E} \absol{\frac{1}{t}
\int_{(0, t)} \bfF^{(\beta)}_2f(\phi_s(y))\diff s
-\bfF^{(\beta)}_2f(y)}\\
&\quad+\lim_{t\downarrow0}\sup_{y\in E} \absol{\frac{1}{t}
\int_{(0, t)} \rbra{\exp\cbra{\int_0^s\nu(\phi_u(y), (-\infty,0))\diff u}-1}\bfF^{(\beta)}_2f(\phi_s(y))\diff s
}\\
&\quad+\lim_{t\downarrow0} e^{tM_\nu}M_\nu\sup_{y\in E}\bE_{(0 ,y)}\sbra{\sup_{s\in [0, t]}|f(\phi_s(y))-f(y)|} \\
&\leq\lim_{t\downarrow0}\sup_{y\in E} \absol{\sup_{s\in[0, t]}\absol{\bfF^{(\beta)}_2f(\phi_s(y))-
\bfF^{(\beta)}_2f(y)}}
+\lim_{t\downarrow0}\frac{e^{t M_\nu}}{t}\cdot2\norm{f}_\infty M_\nu=0,
\label{40605}
\end{align}
where in the first equality we used Remark \ref{BassSection6} and \eqref{40604}, in the second equality we used Remark \ref{BassSection6}, in the second inequality we used Remark \ref{BassSection6}
and in the last inequality and equality we used \eqref{supupper} for $Y^\prime$. 
By 
\eqref{40602}, \eqref{40603}, \eqref{40601new}, \eqref{40604} and \eqref{40605}, the proof is complete. 
\end{proof}
For every $q>0$ and $x\in[0, \infty)$, it follows from the form of $\bfL^{(q)}_x$ that it is a bounded linear operator with norm at most $\frac{1}{\un{d}}$. Moreover, since $\{\bfH^{(q)}_a:a\geq 0\}$ is extended to be a $C_0$-group, there exist constant $K_{\bfH^{(q)}}$ and $M_{\bfH^{(q)}}$ such that the norm of $\bfH^{(q),-1}_x$ 
is bounded by $ K_{\bfH^{(q)}}e^{M_{\bfH^{(q)}}x}$ for every $x\in[0,\infty)$. 
Thus, 
\begin{align}
\text{the norm of }\bfW^{(q)}_x\text{ is bounded by }\frac{K_{\bfH^{(q)}}}{\un{d}}e^{M_{\bfH^{(q)}}x}\text{ for }q> 0
\text{ and }x\geq 0. 
\label{Wnormbound}
\end{align}
So, we can define, for $q> 0$ and $\beta>M_{\bfH^{(q)}}$, 
\begin{align}
\bfV^{(q)}_\beta f( y)&:=\int_\bR e^{-\beta x}\bfW^{(q)}_{x} f(y) \diff x
\qquad
f\in C_0(E), ~
y\in E.
\end{align}
Note that ${\bfV}^{(q)}_\beta f \in C_0 (E)$ by the dominated convergence theorem. 
\begin{Lem}\label{Lem408}
For $q>0$, $\beta>M_{\bfH^{(q)}}$, and $f\in C_0(E)$, we have 
\begin{align}
\lim_{t\downarrow 0}\sup_{y\in E}
\absol{
\int_0^\infty e^{-\beta x}\frac{\bE_{(x,y)} \sbra{\bfW^{(q)}_{X_t}f(Y_t)}-\bfW^{(q)}_xf(y)}{t}\diff x-
q {V}^{(q)}_\beta f(y)}=0 . \label{40702}
\end{align}
\end{Lem}
\begin{proof}
We fix $(x, y)\in [0, \infty)\times E$ and $t\in (0,1)$. 
We have, 
\begin{align}
&\frac{\bE_{(x,y)} \sbra{\bfW^{(q)}_{X_t}f(Y_t)}-\bfW^{(q)}_xf(y)}{t}\\
&=
\frac{e^{-qt}\bE_{(x,y)} \sbra{\bfW^{(q)}_{X_t}f(Y_t)}-\bfW^{(q)}_xf(y)}{t}+\frac{1-e^{-qt}}{t}\bE_{(x,y)} \sbra{\bfW^{(q)}_{X_t}f(Y_t)}. 
\label{407001}
\end{align}
Since the drift of $X$ is less than $\bar{d}$, the process $X$, started at $x$, cannot hit $(x+t\bar{d}, \infty )$ by time $t$.
So, by Markov property at $t$ and Theorem \ref{Thm401}, we have, for $t\in(0, 1)$,
\begin{align}
&\frac{e^{-qt}\bE_{(x,y)} \sbra{\bfW^{(q)}_{X_t}f(Y_t)}-\bfW^{(q)}_xf(y)}{t}\\
&=\frac{e^{-qt}\bE_{(x,y)} \sbra{\bfQ^{(q),[0,{x+\bar{d}}]}_{X_t}\bfW^{(q)}_{x+\bar{d}}f(Y_t)}-\bfQ^{(q),[0,{x+\bar{d}}]}_x\bfW^{(q)}_{x+\bar{d}}f(y)}{t}
\\
&=\frac{\bE_{(x, y)}\sbra{e^{-q\tau^+_{x+\bar{d}}}\bfW^{(q)}_{x+\bar{d}}f(Y_{\tau^+_{x+\bar{d}}});t<\tau^+_{x+\bar{d}}<\tau^-_0}-\bfQ^{(q),[0,{x+\bar{d}}]}_x\bfW^{(q)}_{x+\bar{d}}f(y)}{t}. \label{407002a}
\end{align}
In addition, $\bE_{(x, y)}\sbra{e^{-q\tau^+_{x+\bar{d}}}\bfW^{(q)}_{x+\bar{d}}f(Y_{\tau^+_{x+\bar{d}}});t<\tau^+_{x+\bar{d}}<\tau^-_0}$ and $\bfQ^{(q),[0,{x+\bar{d}}]}_x\bfW^{(q)}_{x+\bar{d}}f(y)$ differ only on the event that $X$ enters $(-\infty,0)$ and subsequently returns to $[0,\infty)$ by time $t$.
So, by \eqref{Wnormbound}, we have 
\begin{align}
\int_0^\infty e^{-\beta x}&\absol{\frac{\bE_{(x, y)}\sbra{e^{-q\tau^+_{x+\bar{d}}}\bfW^{(q)}_{x+\bar{d}}f(Y_{\tau^+_{x+\bar{d}}});t<\tau^+_{x+\bar{d}}<\tau^-_0}-\bfQ^{(q),[0,{x+\bar{d}}]}_x\bfW^{(q)}_{x+\bar{d}}f(y)}{t}}\diff x\\
&  \leq
\int_0^\infty e^{-\beta x}\frac{\bP_{(x, y)}\sbra{T_0<t}}{t} \norm{\bfW^{(q)}_{x+\bar{d}}f}_\infty\diff x\\
&<\frac{K_{\bfH^{(q)}}e^{M_{\bfH^{(q)}}{\bar{d}}}}{\un{d}}
\norm{f}_\infty
\int_0^\infty e^{-(\beta-M_{\bfH^{(q)}}) x}\frac{\bP_{(x, y)}\sbra{T_0<t}}{t}  \diff x. \label{407002}
\end{align}
For $T_0\leq t$ to occur, $X$ must have at least one jump by time $t$. 
The probability of two or more jumps occurring in $(0,t)$ is at most $1-e^{-(M_\nu+M_\ttn)t}(1+(M_\nu+M_\ttn)t)$ by Remark \ref{BassSection6}, which tends to zero after division by $t$ as $t\downarrow 0$. 
Therefore, to analyze $\frac{\bP_{(x, y)}\sbra{T_0<t}}{t}$, it suffices to consider the first jump. 
Since $X$ starts from $x$ and its drift rate is less than $\bar{d}$, for $X$ to fall below $0$ in a single jump and subsequently return to $[0,\infty)$ by time $t$, the jump size must lie in $[-x-t\bar{d},-x)$. 
By the above discussion and since the first jump time and Remark \ref{BassSection6}, and the jump size depend only on $y$, we can use the Fubini's theorem and we have 
\begin{align}
&\int_0^\infty e^{-(\beta-M_{\bfH^{(q)}}) x}\frac{\bP_{(x, y)}\sbra{T_0<t}}{t}  \diff x\\
&\leq
\frac{1}{t}\int_0^\infty e^{-(\beta-M_{\bfH^{(q)}}) x}
\bigg{(}\int_{(0, t)}\bigg{(} 
{\tilde{\ttn}(\phi_s(y), [-x-t\bar{d},-x))}\\
&\qquad\qquad+
\tilde{\rho}(\phi_s(y), [-x-t\bar{d},-x))
\bigg{)} \bP_{(0, y)}\rbra{T^X_1\in \diff s}\bigg{)}
 \diff x\\
&\leq
\frac{1}{t(\beta-M_{\bfH^{(q)}})}
\bigg{(}\int_{(0, t)}\bigg{(} 
\int_{(-\infty, 0)}\rbra{e^{(\beta-M_{\bfH^{(q)}}) (x+\bar{d}t)}-e^{(\beta-M_{\bfH^{(q)}})x}}
{\tilde{\ttn}(\phi_s(y), \diff x)}
\\
&\qquad\qquad+
\int_{(-\infty, 0)}\rbra{e^{(\beta-M_{\bfH^{(q)}}) (x+\bar{d}t)}-e^{(\beta-M_{\bfH^{(q)}})x}} 
\tilde{\rho}(\phi_s(y), \diff x)
\bigg{)} \bP_{(x, y)}\rbra{T^X_1\in \diff s}\bigg{)}\\
&\leq \frac{1-e^{-(M_\nu+M_\ttn)t}}{t(\beta-M_{\bfH^{(q)}})} \cdot 2 \rbra{e^{(\beta-M_{\bfH^{(q)}}) \bar{d}t}-1}.
\label{407003a}
\end{align}
From \eqref{407002a}, \eqref{407002} and \eqref{407003a}, we have 
\begin{align}
\lim_{t\downarrow0}
\sup_{y\in E}\absol{\int_0^\infty e^{-\beta x}\frac{e^{-qt}\bE_{(x,y)} \sbra{\bfW^{(q)}_{X_t}f(Y_t)}-\bfW^{(q)}_xf(y)}{t}\diff x}=0. \label{407001b}
\end{align}
From 
\eqref{407001b}, we also have 
\begin{align}
\lim_{t\downarrow0}&\sup_{y\in E}\absol{\int_0^\infty e^{-\beta x}\frac{1-e^{-qt}}{t}\bE_{(x,y)} \sbra{\bfW^{(q)}_{X_t}f(Y_t)} \diff x 
-{\bfV}^{(q)}_\beta f(y)}\\
&\leq\lim_{t\downarrow0}\sup_{y\in E}
\absol{\int_0^\infty e^{-\beta x}e^{-qt}\bE_{(x,y)} \sbra{\bfW^{(q)}_{X_t}f(Y_t)} \diff x 
-{\bfV}^{(q)}_\beta f(y)}\\
&+\lim_{t\downarrow0}\absol{\frac{1-e^{-qt}}{t}-e^{-qt}}
\cdot\absol{\int_0^\infty e^{-\beta x}\bE_{(x,y)} \sbra{\bfW^{(q)}_{X_t}f(Y_t)} \diff x}=0. \label{407002b}
\end{align}
By \eqref{407001}, \eqref{407001b} and \eqref{407002b}, the proof is complete. 
\end{proof}
\begin{Lem}\label{Lem409}
For $q>0$, $\beta>M_{\bfH^{(q)}}$, and $f\in C_0(E)$, we have 
\begin{align}
\lim_{t\downarrow 0}\sup_{y\in E}
\absol{
\int_{-\infty}^0 e^{-\beta x}\frac{\bE_{(x,y)} \sbra{\bfW^{(q)}_{X_t}f(Y_t)}}{t}\diff x-
f(y)}=0 . \label{40702}
\end{align}

\end{Lem}
\begin{proof}
We can split as for $y\in E$, 
\begin{align}
\int_{-\infty}^0 e^{-\beta x}\frac{\bE_{(x,y)} \sbra{\bfW^{(q)}_{X_t}f(Y_t)}}{t}\diff x&
=\int_{-\infty}^0 e^{-\beta x}\frac{\bE_{(x,y)} \sbra{\bfW^{(q)}_{X_t}f(Y_t); T^X_1>t}}{t}\diff x\\
&~+\int_{-\infty}^0 e^{-\beta x}\frac{\bE_{(x,y)} \sbra{\bfW^{(q)}_{X_t}f(Y_t); T^X_1\leq t}}{t}\diff x.
\label{410last1}
\end{align}
Since the drift rate of $X$ is bounded above by $\bar{d}$, $X$ can move to the right by at most $\bar{d}t$ by time $t$. Hence, if $x<at$, then $\bE_{(x,\cdot)} \sbra{\bfW^{(q)}_{X_t}f(Y_t)}=0$. 
In addition, by Remark \ref{BassSection6}, $\bP_{(x, y)}(T^X_1\leq t) \leq 1-e^{-(M_\nu+M_\ttn)t}$, so combining with \eqref{Wnormbound}, we have, for $y\in E$ 
\begin{align}
\int_{-\infty}^0 e^{-\beta x}\frac{\bE_{(x,y)} \sbra{\bfW^{(q)}_{X_t}f(Y_t); T^X_1\leq t}}{t}\diff x
&=\int_{-\bar{d}t}^0 e^{-\beta x}\frac{\bE_{(x,y)} \sbra{\bfW^{(q)}_{X_t}f(Y_t); T^X_1\leq t}}{t}\diff x\\
&\leq \frac{\bar{d}t}{t}(1-e^{-(M_\nu+M_\ttn)t})\frac{K_{\bfH^{(q)}}}{\un{d}}e^{M_{\bfH^{(q)}}\bar{d}t}\norm{f}_\infty. 
\label{410last2}
\end{align}
If no jump occurs before time $t$, then by \eqref{L--K}, the increment of $X$ is given by $\psi_y(t)$ under $\bP_{(\cdot, y)}$, where 
\begin{align}
\psi_y(t):=\int_0^t \ttd(\phi_s(y))\,\diff s. 
\end{align}
By an argument similar to the one above, we have
\begin{align}
\int_{-\infty}^0 e^{-\beta x}&\frac{\bE_{(x,y)} \sbra{\bfW^{(q)}_{X_t}f(Y_t); T^X_1>t}}{t}\diff x\\
&=\int_{-\psi_y(t)}^0 e^{-\beta x}\frac{\bE_{(x,y)} \sbra{\bfW^{(q)}_{X_t}f(Y_t); T^X_1>t}}{t}\diff x. 
\label{410last3}
\end{align}
We define $\tau^t_x:=\inf\{s>t: X_t<x\}$ for $x\in \bR$. 
For $a>\bar{d}t$, we have 
\begin{align}
&\int_{-\psi_y(t)}^0 e^{-\beta x }\frac{\bE_{(x, y)}\sbra{\bfW^{(q)}_{X_t}f(Y_t) ; T^X_1>t}}{t}\diff x\\
&=\int_{-\psi_y(t)}^0 e^{-\beta x }\frac{\bE_{(x, y)}\sbra{\bfQ^{(q),[0, a]}_{X_t}
\bfW^{(q)}_{a}f(Y_t) ; T^X_1>t}}{t}\diff x\\
&=\int_{-\psi_y(t)}^0 e^{qt-\beta x }\frac{\bE_{(x, y)}\sbra{e^{-q\tau^+_{a}}
\bfW^{(q)}_{a}f(Y_{\tau^+_{a}}) ; t< T^X_1, \tau^+_{a}<\tau^t_0}}{t}\diff x\\
&=\frac{1}{t}\int_{-\psi_y(t)}^0 e^{qt-\beta  x}\bE_{(0, y)}\sbra{e^{-q\tau^+_{a-x}}
\bfW^{(q)}_af(Y_{\tau^+_{a-x}}) ; t< T^X_1, \tau^+_{a-x}<\tau^t_{-x}}\diff x \\
&=\frac{1}{t}\int_{-\psi_y(t)}^0 e^{qt-\beta  x}
\bfQ^{(q), [0, a-x]}_0\bfW^{(q)}_af(y)
\diff x \\
&~~-\frac{1}{t}\int_{-\psi_y(t)}^0 e^{qt-\beta  x}\bE_{(0, y)}\sbra{e^{-q\tau^+_{a-x}}
\bfW^{(q)}_af(Y_{\tau^+_{a-x}}) ;  T^X_1 \leq t, \tau^+_{a-x}<\tau^-_0}\diff x \\
&~~-\frac{1}{t}\int_{-\psi_y(t)}^0 e^{qt-\beta  x}\bE_{(0, y)}
\Big{[}e^{-q\tau^+_{a-x}}
\bfW^{(q)}_af(Y_{\tau^+_{a-x}}) ; t<T^X_1,\tau^t_{-x}\leq \tau^+_{a-x}<\tau^-_0\Big{]}\diff x,
\label{410last4} 
\end{align}
where in the first equality we used Theorem \ref{Thm401}, in the second equality we used the Markov property at $t$ and the fact that $X$ cannot go above $a$ before $t$ and cannot below $0$ without jumps. 
By Assumption \ref{Ass203} (ii) and \eqref{upestemate1} for $Y^\prime$, we have 
\begin{align}
\lim_{t\downarrow0}\sup_{y\in E}\absol{\frac{\psi_y (t)}{t}-\ttd(y)}=0.  \label{driftlim}
\end{align}
By Remark \ref{BassSection6}, Assumption \ref{Ass203} (ii) and \eqref{upestemate1} for $Y^\prime$, we have, for $\varepsilon>0$ and $g\in C_0(E)$, 
\begin{align}
\lim_{\varepsilon\downarrow0}\norm{\bfQ^{(q),[0,a+\varepsilon]}_{a}g-g}_\infty
\leq&\lim_{\varepsilon\downarrow0}\sup_{y\in E}\absol{\bE_{(a, y)}\sbra{e^{-q\tau^+_{a+\varepsilon}}g(Y_{\tau^+_{a+\varepsilon}})-g(y);\tau^+_{a+\varepsilon}<\tau^-_0, T^X_1>\frac{\varepsilon}{\un{d}}}}\\
+&\lim_{\varepsilon\downarrow0}\sup_{y\in E}\absol{\bE_{(a, y)}\sbra{e^{-q\tau^+_{a+\varepsilon}}g(Y_{\tau^+_{a+\varepsilon}})-g(y);\tau^+_{a+\varepsilon}<\tau^-_0, T^X_1\leq \frac{\varepsilon}{\un{d}}}}\\
\leq&\lim_{\varepsilon\downarrow0}\sup_{y\in E}\bE_{(a, y)}\sbra{\sup_{s\in[0, \frac{\varepsilon}{\un{d}}]}\absol{e^{-qs}g(Y^\prime_s)-g(y)}}\\
+&\lim_{\varepsilon\downarrow0}2\norm{g}_\infty \rbra{1-\exp\rbra{-(M_\nu+M_\ttn)\frac{\varepsilon}{\un{d}}}}=0, 
\end{align}
therefore, by combining with the fact that the norm of $\bfQ^{(q),[0,a]}_0$ is no more than $1$,
\begin{align}
\lim_{\varepsilon\downarrow0}&\norm{\bfQ^{(q),[0,a+\varepsilon]}_0\bfW^{(q)}_af-\bfQ^{(q),[0,a]}_0\bfW^{(q)}_af}_\infty\\
&=\lim_{\varepsilon\downarrow0}\norm{\bfQ^{(q),[0,a]}_0\bfQ^{(q),[0,a+\varepsilon]}_a\bfW^{(q)}_af-\bfQ^{(q),[0,a]}_0\bfW^{(q)}_af}_\infty\\
&=\lim_{\varepsilon\downarrow0}\norm{\bfQ^{(q),[0,a]}_0(\bfQ^{(q),[0,a+\varepsilon]}_a\bfW^{(q)}_a-\bfW^{(q)}_a)f}_\infty\\
&=\lim_{\varepsilon\downarrow0}\norm{(\bfQ^{(q),[0,a+\varepsilon]}_{a}\bfW^{(q)}_a-
\bfW^{(q)}_a)f}_\infty=0. \label{driftlim2}
\end{align}
By \eqref{driftlim}, \eqref{driftlim2} and \eqref{scaleinverse}, we have
\begin{align}
\lim_{t\downarrow 0}&\sup_{y\in E}\absol{\frac{1}{t}\int_{-\psi_y(t)}^0 e^{qt-\beta  x}
\bfQ^{(q), [0, a-x]}_0\bfW^{(q)}_af(y)
\diff x-f(y)}\\
&=\lim_{t\downarrow 0}\sup_{y\in E}\absol{\frac{\psi_y(t)}{t}\int_{-1}^0 e^{qt-\beta \psi_y(t) x}
\bfQ^{(q),[0,a-\psi_y(t)x]}_0\bfW^{(q)}_af(y)
\diff x-f(y)}\\
&=
\sup_{y\in E}\absol{\ttd(y) \bfQ^{(q),[0,a]}_0\bfW^{(q)}_af(y)-f(y)}
=\lim_{t\downarrow 0}\norm{\ttE^{(q)}_a\bfW^{(q)}_af-f}_\infty=0. 
\label{410last5}
\end{align}
By Assumption \ref{Ass203} (ii), \eqref{Wnormbound} and Remark \ref{BassSection6}, we have, 
for $t\in (0, \frac{1}{\bar{d}})$,
\begin{align}
\sup_{y\in E}&\absol{
\frac{1}{t}\int_{-\psi_y(t)}^0 e^{qt-\beta  x}\bE_{(0, y)}\sbra{e^{-q\tau^+_{a-x}}
\bfW^{(q)}_af(Y_{\tau^+_{a-x}}) ;  T^X_1 \leq t, \tau^+_{a-x}<\tau^-_0}\diff x
}
\\
&\leq e^{qt}\bar{d}\cdot\frac{K_{\bfH^{(q)}}}{\un{d}}e^{M_{\bfH^{(q)}}a}\norm{f}_\infty\sup_{y\in E}\bP_{(0, y)}(T^X_1\leq t)\\
&\leq  e^{qt}\frac{\bar{d}K_{\bfH^{(q)}}}{\un{d}}e^{M_{\bfH^{(q)}}a}\norm{f}_\infty(1-e^{-(M_\nu+M_\ttn)t}), 
\label{410last6}
\end{align}
and
\begin{align}
&\sup_{y\in E}\Big{|}\frac{1}{t}\int_{-\psi_y(t)}^0 e^{qt- x}\bE_{(0, y)}
\Big{[}e^{-q\tau^+_{a-x}}
\bfW^{(q)}_af(Y_{\tau^+_{a-x}}) ; t<T^X_1,\tau^t_{-x}\leq \tau^+_{a-x}<\tau^-_0\Big{]}\diff x\Big{|}\\
&\leq e^{qt}\frac{K_{\bfH^{(q)}}}{\un{d}t}e^{M_{\bfH^{(q)}}a}\norm{f}_\infty\sup_{y\in E}
\int_{-\psi_y(t)}^0e^{-\beta x}\bP_{(0, y)}\rbra{t<T^X_1, \tau^-_{-x}<  \tau^-_0\land\tau^+_{a+1}}\diff x \\
&\leq e^{qt}\frac{K_{\bfH^{(q)}}}{\un{d}t}e^{M_{\bfH^{(q)}}a}\norm{f}_\infty\sup_{y\in E}
\int_{-\psi_y(t)}^0 e^{-\beta x}
\sum_{n\in\bN} \bE_{(0, y)}\Big{[} \tilde{\rho}(Y_{T^X_n-}, [-X_{T^X_n-}  ,-X_{T^X_n-}-x))    \\
&\qquad \qquad+\tilde{\ttn}(Y_{T^X_n-}, [-X_{T^X_n-}  ,-X_{T^X_n-}-x)) ; t<T^X_n\leq\tau^-_0 \land \tau^+_{a+1}\Big{]}
\diff x\\
&= e^{qt}\frac{K_{\bfH^{(q)}}}{\un{d}t\beta}e^{M_{\bfH^{(q)}}a}\norm{f}_\infty\sup_{y\in E}
\sum_{n\in\bN} \bE_{(0, y)}\Big{[}\int_{(-X_{T^X_n-}, -X_{T^X_n-}+\psi_y(t))}\rbra{e^{\beta\psi_y(t)}-e^{\beta(X_{T^X_n-}+x)}} \tilde{\rho}(Y_{T^X_n-}, \diff x)    \\
&\qquad +
\int_{(-X_{T^X_n-}, -X_{T^X_n-}+\psi_y(t))}\rbra{e^{\beta\psi_y(t)}-e^{\beta(X_{T^X_n-}+x)}} \tilde{\ttn}(Y_{T^X_n-}, \diff x)  ; t<T^X_n\leq\tau^-_0 \land \tau^+_{a+1}\Big{]}\\
&\leq e^{qt}\frac{\bar{d}tK_{\bfH^{(q)}}}{\un{d}t\beta}e^{M_{\bfH^{(q)}}a}\norm{f}_\infty\rbra{e^{\beta\bar{d}t}-1}\sup_{y\in E}
\sum_{n\in\bN}\bP_{(0, y)}\rbra{T^X_n\leq\tau^-_0 \land \tau^+_{a+1}}. \label{410last7}
\end{align}
Since the drift size of $X$ is more than $\un{d}$ and by Remark \ref{BassSection6}, we have, for $(x, y)\in(0 ,a)\times E$ and $n\in \bN$, 
\begin{align}
\bP_{(x, y)}\rbra{T^X_{1}\leq\tau^-_0 \land \tau^+_{a+1}}\leq \rbra{1-e^{-(M_\nu+M_\ttn)\frac{a+1}{\bar{d}}}}
\end{align}
and so
\begin{align}
\bP_{(0, y)}\rbra{T^X_n\leq\tau^-_0 \land \tau^+_{a+1}}
&\leq \bE_{(0, y)}\sbra{\bP_{(X_{T^X_{n-1}}, Y_{T^X_{n-1}})}\rbra{T^X_{1}\leq\tau^-_0 \land \tau^+_{a+1}};T^X_{n-1}\leq\tau^-_0 \land \tau^+_{a+1}}\\
&\leq\bP_{(0, y)}\rbra{T^X_{n-1}\leq\tau^-_0 \land \tau^+_{a+1}}\rbra{1-e^{-(M_\nu+M_\ttn)\frac{a+1}{\bar{d}}}}, 
\end{align}
therefore, we have 
\begin{align}
\sup_{y\in E}
\sum_{n\in\bN}\bP_{(0, y)}\rbra{T^X_n\leq\tau^-_0 \land \tau^+_{a+1}}
\leq e^{(M_\nu+M_\ttn)\frac{a+1}{\bar{d}}}-1. \label{410last8}
\end{align}
By \eqref{410last1}, \eqref{410last2}, \eqref{410last3}, \eqref{410last4}, \eqref{410last5}, \eqref{410last6}, \eqref{410last7} and \eqref{410last8}, the proof is complete. 
\end{proof}
\begin{Lem}\label{Lem410}
For $q>0$, $\beta>M_{\bfH^{(q)}}$ and $f\in C_0(E)$, we have 
\begin{align}
\lim_{t\downarrow 0}\sup_{y\in E}
\absol{
\frac{\bE_{(0,y)} \sbra{e^{\beta X_t}{\bfV}^{(q)}_\beta f(Y_t)}-{\bfV}^{(q)}_\beta f(y)}{t}-
q {\bf{V}}^{(q)}_\beta f(y)-f(y)}.
\end{align}
\end{Lem}
\begin{proof}
By Fubini's theorem and Lemmas \ref{Lem408} and \ref{Lem409}, we have, for $q>0$, $\beta>M_{\bfH^{(q)}}$ and $f\in C_0(E)$,
\begin{align}
&\lim_{t\downarrow 0}\sup_{y\in E}
\absol{
\frac{\bE_{(0,y)} \sbra{e^{\beta X_t}{\bfV}^{(q)}_\beta f(Y_t)}-{\bfV}^{(q)}_\beta f(y)}{t}-
q {\bf{V}}^{(q)}_\beta f(y)-f(y)}\\
&=\lim_{t\downarrow 0}\sup_{y\in E}
\absol{
\frac{ \bE_{(0,y)} \sbra{\int_\bR e^{-\beta (x-X_t)} \bfW^{(q)}_{x}f(Y_t)\diff x}-{\bfV}^{(q)}_\beta f(y)}{t}-
q {\bf{V}}^{(q)}_\beta f(y)-f(y)}\\ 
&=\lim_{t\downarrow 0}\sup_{y\in E}
\absol{
\int_\bR e^{-\beta x}\frac{\bE_{(x,y)} \sbra{\bfW^{(q)}_{X_t}f(Y_t)}-\bfW^{(q)}_xf(y)}{t}\diff x-
q {\bf{V}}^{(q)}_\beta f(y)-f(y)}=0 . 
\end{align}
The proof is complete. 
\end{proof}
\begin{proof}[Proof of Theorem \ref{Thm300}]
By Lemmas \ref{Lem406} and \ref{Lem410}, 
for $q>0$, $\beta>M_{\bfH^{(q)}}$ and $f\in C_0(E)$, 
there exists a function $\bfg^{(q)}_\beta \in C_0(E)$ such that 
\begin{align}
\lim_{t\downarrow0}\sup_{y\in E}\absol{\frac{\bE_{(0,y)} \sbra{{\bf{V}}^{(q)}_\beta f( Y_t) -{\bf{V}}^{(q)}_\beta f(y); t<T^Y_1}}{t}
-\bfg^{(q)}_\beta (y)}=0. \label{Yprimelim}
\end{align}
Since $Y$ evolves as the deterministic process $Y^\prime$ until its first jump and by Remark \ref{BassSection6}, we have, for $q>0$, $\beta>M_{\bfH^{(q)}}$, $f\in C_0(E)$ and $y\in E$, 
\begin{align}
\frac{\bE_{(0,y)} \sbra{{\bf{V}}^{(q)}_\beta f( Y_t) -{\bf{V}}^{(q)}_\beta f(y); t<T^Y_1}}{t}
=\frac{{\bf{V}}^{(q)}_\beta f( \phi_t(y)) -{\bf{V}}^{(q)}_\beta f(y)}{t}
\bP_{(0,y)}(t<T^Y_1)
\end{align}
and so
\begin{align}
&\sup_{y\in E}\absol{\frac{{\bf{V}}^{(q)}_\beta f( \phi_t(y)) -{\bf{V}}^{(q)}_\beta f(y)}{t}-\frac{\bE_{(0,y)} \sbra{{\bf{V}}^{(q)}_\beta f( Y_t) -{\bf{V}}^{(q)}_\beta f(y); t<T^Y_1}}{t}}\\
&=
\sup_{y\in E}\frac{\bE_{(0,y)} \sbra{{\bf{V}}^{(q)}_\beta f( Y_t) -{\bf{V}}^{(q)}_\beta f(y); t<T^Y_1}}{t}
\rbra{\frac{1}{\bP_{(0,y)}(t<T^Y_1)}-1}\\
&\leq \sup_{y\in E}\frac{\bE_{(0,y)} \sbra{{\bf{V}}^{(q)}_\beta f( Y_t) -{\bf{V}}^{(q)}_\beta f(y); t<T^Y_1}}{t}
(e^{M_\nu t}-1). \label{sa}
\end{align}
By \eqref{Yprimelim} and \eqref{sa}, we have, for $q>0$, $\beta>M_{\bfH^{(q)}}$ and $f\in C_0(E)$, 
\begin{align}
\lim_{t\downarrow0}\sup_{y\in E}\absol{\frac{\bE_{(0,y)} \sbra{{\bf{V}}^{(q)}_\beta f( Y_t)} -{\bf{V}}^{(q)}_\beta f(y)}{t}
-\bfg^{(q)}_\beta (y)}=0, \label{Yprimelim2}
\end{align}
and ${\bf{V}}^{(q)}_\beta f \in D(\cA_Y)$. 
By Lemmas \ref{Lem406} and \ref{Lem410}, \eqref{Dbetadecom}, \eqref{Yprimelim} and \eqref{Yprimelim2}, we have, for $q>0$, $\beta>M_{\bfH^{(q)}}$ and $f\in C_0(E)$, 
\begin{align}
\bfF^{(\beta)}{\bf{V}}^{(q)}_\beta f(y)=q {\bf{V}}^{(q)}_\beta f(y)+f(y),\qquad y\in E. 
\end{align}
It remains to show that, for every$q\geq 0$, there exists a sufficiently large $\beta \geq 0$ such that $\bfF^{(\beta)}-q\bfI$ is invertible.
Here, we follow the argument in \cite[p.159]{EngNag2000}.
For $q\geq 0$, we write $\tilde{\bfR}^{(q)}_{Y^\prime}$ for the $q$-resolvent operator of $Y^\prime$. 
For $q, \beta > 0$ and $a := \frac{q}{\beta}+\bar{d}$, we have 
\begin{align}
\bfF^{(\beta)}-q\bfI
&=(\beta a-q)\bfI+\rbra{(\bff^{(\beta)}_1-\beta a)\bfI+\bfF^{(\beta)}_2+\cA_{Y^\prime}}\\
&=\rbra{\bfI+\rbra{(\bff^{(\beta)}_1-\beta a)\bfI+\bfF^{(\beta)}_2}\tilde{\bfR}^{(\beta a-q)}_{Y^\prime}}\rbra{(\beta a-q)\bfI+\cA_{Y^\prime}},
\end{align}
where in the last equality we used the fact that $\tilde{\bfR}^{(\beta a-q)}_{Y^\prime}\rbra{(\beta a-q)\bfI+\cA_{Y^\prime}}=\bfI$. 
By the forms of $\bff^{(\beta)}_1$ and $\bfF^{(\beta)}_2$, the norm of $(\bff^{(\beta)}_1-\beta a)\bfI+\bfF^{(\beta)}_2$ is no more than $(a-\bar{d})\beta+M_\ttn+2M_\nu$. 
In addition, the norm of $\tilde{\bfR}^{(\beta a-q)}_{Y^\prime}$ is no more than $\frac{1}{\beta a-q}$. 
Thus, the norm of $\rbra{(\bff^{(\beta)}_1-\beta a)I+\bfF^{(\beta)}_2}\tilde{\bfR}^{(\beta a-q)}_{Y^\prime}$ is  no more than 
\begin{align}
\frac{(a-\bar{d})\beta+M_\ttn+2M_\nu}{\beta a-q}=
\frac{q+M_\ttn+2M_\nu}{\beta\bar{d}}, 
\end{align}
which is less than $1$ for $\beta >\frac{q+M_\ttn+2M_\nu}{\bar{d}}$. 
So, for $q, \beta >0$ with $\beta >\frac{q+M_\ttn+2M_\nu}{\bar{d}}$, the operator $I+\rbra{(\bff^{(\beta)}_1-\beta a)I+\bfF^{(\beta)}_2}\tilde{\bfR}^{(\beta a-q)}_{Y^\prime}$ is invertible, and thus $\bfF^{(\beta)}-q\bfI$ is invertible. 
The proof is complete. 
\end{proof}

\section*{Acknowledgments}
The author was supported by JSPS KAKENHI grant no. JP21K13807 and JSPS Open Partnership Joint Research Projects grant no. JPJSBP120209921. 
In addition, from February 2023 to February 2025, the author stayed at Centro de Investigaci\'on en Matem\'aticas (CIMAT), Mexico, as a JSPS Overseas Research Fellow. The author is grateful for the stimulating research environment at CIMAT, which helped develop the important ideas of this paper.
\section*{Declaration of AI use}
In preparing this paper, the author used ChatGPT to search for known mathematical results, commonly used techniques, and relevant literature. 
ChatGPT was also used to translate Japanese text into English and to assist in checking the correctness of a small part of the manuscript. 
All results provided by ChatGPT have been carefully verified by the author, who takes full responsibility for the content of this paper.

\bibliographystyle{jplain}
\bibliography{NOBA_references_07}

\end{document}